\documentclass{amsart}

\usepackage{amssymb,color}
\usepackage{amsfonts}
\usepackage{amsmath}
\usepackage{euscript}
\usepackage{enumerate}
\usepackage{graphics}
\usepackage{pdfsync}

\newtheorem{theorem}{Theorem}[section]
\newtheorem{lemma}[theorem]{Lemma}
\newtheorem{note}[theorem]{Note}
\newtheorem{prop}[theorem]{Proposition}
\newtheorem{cor}[theorem]{Corollary}

\newtheorem*{Theorem1'}{Theorem 1'}

\theoremstyle{definition}
\newtheorem{definition}[theorem]{Definition}
\newtheorem{example}[theorem]{Example}

\theoremstyle{remark}

\numberwithin{equation}{section}

\newcommand \g{{\mathfrak g}}
\newcommand \h{{\mathfrak h}}
\newcommand  \s{{\mathfrak s}}
\renewcommand \r{{\mathfrak r}}

\newcommand \e{{\mathfrak e}}

\renewcommand \t{{\mathfrak t}}

\newcommand \Hom{{\mathrm {Hom}}}

\newcommand \gl{{\mathfrak {gl}}}
\renewcommand \sl{{\mathfrak {sl}}}
\newcommand \so{{\mathfrak {so}}}
\newcommand \sy{{\mathfrak {sp}}}
\newcommand \C{{\mathbb C}}
\newcommand \Z{{\mathbb Z}}

\newcommand \la{{\lambda}}

\newcommand \al{{\alpha}}

\begin{document}

\title[Racah-Wigner $6j$-symbol and uniserial $\sl(2)\ltimes V(m)$-modules]
{A new interpretation of the Racah-Wigner $6j$-symbol
and the classification of uniserial $\sl(2)\ltimes V(m)$-modules}

\author{Leandro Cagliero}
\address{CIEM-CONICET, FAMAF-Universidad Nacional de C\'ordoba, C\'ordoba, Argentina.}
\curraddr{Math. Department, MIT, Cambridge, MA 02139-4307, USA.}
\email{cagliero@famaf.unc.edu.ar}
\thanks{The first author was supported in part by CONICET and SECYT-UNC grants.}

\author{Fernando Szechtman}
\address{Department of Mathematics and Statistics, Univeristy of Regina, Canada}
\email{fernando.szechtman@gmail.com}
\thanks{The second author was supported in part by an NSERC discovery grant}





\begin{abstract} All Lie algebras and representations will be
assumed to be finite dimensional over the complex numbers. Let
$V(m)$ be the irreducible $\sl(2)$-module with highest weight $m\geq 1$ and consider the perfect Lie algebra
$\g=\sl(2)\ltimes V(m)$. Recall that a $\g$-module is uniserial
when its submodules form a chain. In this paper we classify all
uniserial $\g$-modules. The main family of uniserial $\g$-modules
is actually constructed in greater generality for the perfect Lie
algebra $\g=\s\ltimes V(\mu)$, where $\s$ is a semisimple Lie
algebra and $V(\mu)$ is the irreducible $\s$-module with highest
weight $\mu\neq 0$. The fact that the members of this family are,
but for a few exceptions of lengths 2, 3 and~4, the only uniserial
$\sl(2)\ltimes V(m)$-modules depends in an essential manner on the
determination of certain non-trivial zeros of Racah-Wigner
$6j$-symbol.
\end{abstract}

\maketitle

\section{Introduction}
\label{intro}

All Lie algebras and representations considered in this paper are
assumed to be finite dimensional over the complex numbers.

The problem of classifying all indecomposable modules of a given
Lie algebra (or a family of Lie algebras) is usually hard. Very serious
difficulties are encountered even for Lie algebras of very low
dimensionality.  Two notoriously difficult examples are furnished by the 2-dimensional abelian Lie algebra
(see \cite{GP}, Corollary 1) and the 3-dimensional Euclidean Lie algebra
$\mathfrak{e}(2)=\so(2)\ltimes \C^2$ (see \cite{Sa}, Theorem 4.3).

As is well-known, the classification problem has a satisfactory
answer for the class of all semisimple Lie algebras,
and perfect Lie algebras possess favorable
properties that make them suitable for consideration in this
problem. Indeed, it is precisely the class of Lie algebras that
enjoy an abstract Jordan decomposition \cite{CS}, a crucial tool
needed for the classification of the irreducible representations
of semisimple Lie algebras. Additionally, a Lie algebra $\g$ is
perfect if and only if its solvable radical $\r$ coincides with
its nilpotent radical $[\g,\r]$. Thus, a perfect Lie algebra $\g$
has a Levi decomposition $\g=\s\ltimes \r$, where $\s$ is
semisimple and the solvable radical $\r$ acts trivially on every
irreducible $\g$-module, and hence nilpotently on every
$\g$-module. Further positive features of perfect Lie algebras,
from the representation theory point of view, can be found in
~\S\ref{sec.Matrix recognition}.

 But even for the easiest perfect Lie algebra (other than
semisimple), namely
 $\g=\sl(2)\ltimes\C^2$, the classification of the indecomposable representations
 is far from being achieved (see \cite{DR} and \cite{Pi}).
Therefore a natural approach to this problem is to identify a
distinguished class of indecomposable representations for which
one could expect to obtain a reasonable classification.

This line of research has been followed in a number of papers. For
instance, some authors have considered embeddings of a given $\g$
into a semisimple Lie algebra $\tilde\g$ and considered the
highest weight modules of $\tilde\g$ to construct or classify
indecomposable $\g$-modules obtained by restriction. See, for
instance, \cite{Ca}, \cite{CMS}, \cite{Dd}, \cite{DP}, \cite{DR},
\cite{Pr}. On the other hand, A. Piard \cite{Pi} obtained a
classification of all indecomposable $\g$-modules $V$, where
$\g=\sl(2)\ltimes\r$, $\r=\C^2$ and $V/\r V$ is irreducible.

\medskip

In this paper we focus our attention on the class of uniserial
representations. Recall that a module $V$ is uniserial if its
submodules form a chain. We think that, within this
class, a classification can be achieved for certain families of
Lie algebras and, moreover, that the members of this
class might be viewed as building blocks to understand more
general classes of indecomposable representations.
Indeed, the following two facts support this belief.

First, one of the main results of this paper gives a complete
classification of all uniserial $\g$-modules for $\g=\sl(2)\ltimes
V(m)$, where $V(m)$ is the irreducible $\sl(2)$-module with
highest weight $m\geq 1$. As far as we know, this is the first
time that a structurally defined class of indecomposable modules,
other than the irreducible ones, has been simultaneously
classified for all members of an infinite family of Lie algebras.
Moreover, in contrast to the case of the indecomposable modules of
the abelian and Euclidean Lie algebras, a classification of all
uniserial modules for these and many other solvable Lie algebras
is attained in \cite{CS3}.

Secondly, uniserial modules are also considered  (see \cite{BH},
\cite{HZ}, \cite{HZ2}) as a starting point in terms of
classification and as building blocks of other indecomposable
modules in the case of certain finite dimensional associative
algebras (notice that the Jordan Normal Form Theorem states that
any $\C[x]$-module is a direct sum of uniserial modules). In the
context of Lie algebras, all of Piard's indecomposable
$\sl(2)\ltimes V(1)$-modules mentioned above, as well as further
indecomposable modules for more general perfect Lie algebras, can
be constructed as a series of extensions of uniserial modules (see
\cite{CS4}).

\medskip

A crucial step in the proof of our classification requires the
determination of non-trivial zeros of the (classical) \emph{Racah-Wigner
$6j$-symbol} within certain parameters. The $6j$-symbol is a real
number $\left\{
\begin{matrix}
j_1 \;  j_2 \;  j_3 \\
j_4 \;  j_5 \;  j_6
\end{matrix}
\right\}$ associated to six non-negative half-integer numbers
$j_1$, $j_2$, $j_3$, $j_4$, $j_5$ and $j_6$, originally studied
because it plays a central role in angular momentum theory (see
for instance \cite{CFS}, \cite{Ed}, \cite{RBMW}).

Our results provide a bridge between two related but different lines of
research and we think that very interesting connections of this
kind will appear by considering other families of Lie algebras.

\subsection{Main results} Let $\g$ be a Lie algebra with solvable radical $\r$ and Levi
decomposition $\g=\s\ltimes \r$.

Let $V$ be a $\g$-module and let  $0=V_0\subset V_1\subset
\cdots\subset V_n=V$ be a composition series of $V$. By Lie's
theorem $\r$ acts via scalar operators on each composition factor
$W_i=V_i/V_{i-1}$, $1\leq i\leq n$. In particular every $W_i$ is
an irreducible $\s$-module.

The socle series
$0=\mathrm{soc}^{0}(V)\subset\mathrm{soc}^{1}(V)\subset
\cdots\subset \mathrm{soc}^{m}(V)=V$ is inductively defined by
declaring $\mathrm{soc}^{i}(V)/\mathrm{soc}^{i-1}(V)$ to be socle
of $V/\mathrm{soc}^{i-1}(V)$, that is, the sum of all irreducible
submodules of $V/\mathrm{soc}^{i-1}(V)$, for $1\leq i\leq m$.

 As indicated earlier, $V$ is uniserial if it has only one composition
series, i.e., if the socle series of $V$ has irreducible factors.
A uniserial module is clearly indecomposable. A sequence
$W_1,...,W_n$ of irreducible $\s$-modules will be said to be {\em
admissible} if there is a uniserial $\g$-module with socle factors
$\s$-isomorphic to $W_1,...,W_n$. By considering dual modules, it
is clear that a sequence $W_1,...,W_n$ is admissible if and only
if so is $W^*_n,...,W^*_1$.

In general the classification of uniserial $\g$-modules breaks
down into two steps:

\noindent{\sc Step 1.} Determine all admissible sequences.

\noindent{\sc Step 2.} Given an admissible sequence, find all
uniserial modules giving rise to~it.

As mentioned earlier, perfect Lie algebras are well suited for
consideration in this problem. Indeed, from $\g=\s\ltimes \r$ we
obtain the Levi decomposition $[\g,\g]=\s\ltimes [\g,\r]$. Thus
$\g=[\g,\g]$ if and only if $\r=[\g,\r]$. Here $[\g,\r]=\r\cap
[\g,\g]$ is not only the solvable radical of $[\g,\g]$ but also
the nilpotent radical of $\g$, i.e. the ideal of all $x\in \g$
such that $xV=0$ for every irreducible $\g$-module $V$. It follows
that $\g$ is perfect if and only if $\r$ annihilates every
irreducible $\g$-module. Thus if $\g$ is perfect an irreducible
$\g$-module is nothing but an irreducible $\s$-module annihilated
by $\r$. More generally, if $\g$ is perfect then the terms of the
socle series of $V$ can be intrinsically obtained from $\r$ as
follows: $\mathrm{soc}^{i}(V)/\mathrm{soc}^{i-1}(V)$, for $1\leq
i\leq n$, is simply the is the 0-weight space for the action of
$\r$ on $V/\mathrm{soc}^{i-1}(V)$.

We begin our paper in \S\ref{sec.Matrix recognition} and
\S\ref{secext} by furnishing general criteria to recognize,
construct and classify uniserial modules for perfect Lie algebras with abelian radical.
These results turn out to be fundamental for the rest of the
paper. As a first application we prove in \S\ref{hist} the
following theorem.

\begin{theorem}
\label{uk} Let $\s$ be a non-zero semisimple Lie algebra. Let
$b\in\Z_{\geq 0}$ and let $\lambda$ and $\mu$ be dominant integral
weights of $\s$, where $\mu\neq 0$ and $\mu^*$ is the highest
weight of $V(\mu)^*$. Consider the perfect Lie algebra
$\g=\s\ltimes V(\mu)$. Then, up to isomorphism, there exists one
and only one uniserial $\g$-module, say $Z(\la,b)$, with socle
factors $V(\la),V(\la+\mu^*),...,V(\la+b\mu^*)$.
\end{theorem}

It is then clear that the dual $\g$-module $Z(\la,b)^*$ is, up to
isomorphism, the only uniserial $\g$-module with socle factors
$V(\la^*+b\mu),...,V(\la^*+\mu),V(\la^*)$.

In \S\ref{secex} we exhibit explicit matrix realizations of these
modules for $\g=\sl(2)\ltimes V(m)$ and we spend considerable
effort in \S\ref{ty} to achieve an axiomatic characterization of
$Z(\la,b)$ and $Z(\la,b)^*$, which presents them as a particular
subclass of the class of cyclic indecomposable $\g$-modules.

Other uniserial $\g$-modules are possible for $\g=\s\ltimes
V(\mu)$, as indicated in \S\ref{sec.Other Uniserials}. These
exceptional modules, together with the modules $Z(\la,b)$ and
$Z(\la,b)^*$, comprise all the uniserial modules of
$\g=\sl(2)\ltimes V(m)$. In fact, our main result reads as
follows.

\begin{theorem}
Let $\g=\sl(2)\ltimes V(m)$, where $m\geq 1$. Then, up to a
reversing of the order, the following are the only admissible
sequences for $\g$:

\noindent
\begin{tabular}{ll}
\\[-2mm]
Length 1. & $V(a)$.  \\[2mm]
Length 2. & $V(a),V(b)$, where $a+b\equiv m\mod 2$ and $0\le b-a\leq m\leq a+b$. \\[2mm]
Length 3. & $V(a),V(a+m),V(a+2m)$; or \\[1mm]
      & $V(0),V(m),V(c)$, where $c\equiv 2m \mod 4$ and $c\leq 2m$. \\[2mm]
Length 4. & $V(a),V(a+m),V(a+2m),V(a+3m)$; or \\[1mm]
      & $V(0),V(m),V(m),V(0)$, where $m\equiv 0\mod 4$. \\[2mm]
Length $\geq 5$. &  $V(a),V(a+m),\dots,V(a+s m)$, where $s\geq 4$. \\[2mm]
\end{tabular}

\noindent Moreover, each of these sequences arises from only one
isomorphism class of uniserial $\g$-modules, except for the
sequence $V(0),V(m),V(m),V(0)$, $m\equiv 0\mod 4$. The isomorphism
classes of uniserial $\g$-modules associated to this sequence are
parametrized by the complex numbers.
\end{theorem}

Explicit matrix realizations illustrating this theorem are given
in \S\ref{secex} and \S\ref{sec.Other Uniserials}.

\smallskip

One major step towards the proof of the above theorem is the
determination of all admissible sequences of length 3 for
$\g=\sl(2)\ltimes V(m)$ and this is done in~\S\ref{sec.length3}.
From the results of \S\ref{sec.Matrix recognition} and
\S\ref{secext}, it follows that $V(a),V(b),V(c)$ is an admissible
sequence of length 3 for $\g$ if and only if $V(m)$
enters $V(a)\otimes V(b)$ and $V(b)\otimes V(c)$, and
$\mathcal{L}$ is abelian, where $\mathcal{L}$ is the Lie
subalgebra of $\gl(V)$, with $V=V(a)\oplus V(b)\oplus V(c)$,
generated by $f(r)+g(r)$, $r\in V(m)$, and
$f:V(m)\to\mathrm{Hom}(V(b),V(a))$ as well as
$g:V(m)\to\mathrm{Hom}(V(c),V(b))$ are $\sl(2)$-embeddings.
In
terms of matrices,  $\mathcal{L}$ is the Lie subalgebra of
$\gl(a+b+c+3,\C)$ generated by $
 \left\{\left(\begin{smallmatrix}
    0 & f(r) & 0 \\
    0 & 0 & g(r) \\
    0 & 0 & 0 \\
  \end{smallmatrix}
\right): r\in V(m)\right\}$.
The determination of
all $a$, $b$, $c$ and $m$ for which $\mathcal{L}$ is abelian
requires the following theorem from \S\ref{cal8}.

\begin{theorem}\label{thm.6jIntro}
Let $a,b,c,p,q,k$ be non-negative integers for which there exist
$\sl(2)$-embeddings
\[
 V(k)\overset{f_0}{\longrightarrow}\mathrm{Hom}(V(c),V(a)),
\quad V(k)\overset{f_1}{\longrightarrow}V(p)\otimes V(q),
\]
\[
 V(p)\overset{f_2}{\longrightarrow}\mathrm{Hom}(V(b),V(a)),\quad
 V(q)\overset{f_3}{\longrightarrow}\mathrm{Hom}(V(c),V(b));
\]
and let
$
\left\{\begin{matrix}
\frac{q}{2} \; \frac{k}{2} \; \frac{p}{2} \\[1mm]
\frac{a}{2} \; \frac{b}{2} \; \frac{c}{2}
\end{matrix}
\right\}
$
be the  Racah-Wigner $6j$-symbol associated to them.
If
\[
 f_4=\left\{\begin{matrix}
\frac{q}{2} \; \frac{k}{2} \; \frac{p}{2} \\[1mm]
\frac{a}{2} \; \frac{b}{2} \; \frac{c}{2}
\end{matrix}
\right\}f_0
\]
then, after a suitable normalization of $f_i$, $i=0,1,2,3,4$, the
following diagram of $\sl(2)$-morphisms is commutative, where $g$ sends
$\alpha\otimes\beta\to\alpha\beta$:
\[
 \begin{array}[c]{ccc}
V(k)&\xrightarrow{\hspace{.5cm}\displaystyle f_4\hspace{.5cm}} &\mathrm{Hom}(V(c),V(a))\\[2mm]
{f_1}\Big\downarrow&&\Big\uparrow g\\[2mm]
V(p)\otimes V(q) &\xrightarrow{\;\;\displaystyle
f_2\otimes f_3\;\;}
    &\mathrm{Hom}(V(b),V(a))\otimes
\mathrm{Hom}(V(c),V(b)).
\end{array}
\]
In particular, $V(k)$ appears in the image of
$g(f_2\otimes f_3)$ if and only if $\left\{\begin{matrix}
\frac{q}{2} \; \frac{k}{2} \; \frac{p}{2} \\[1mm]
\frac{a}{2} \; \frac{b}{2} \; \frac{c}{2}
\end{matrix}
\right\}\ne0$.
\end{theorem}
Recall that the (classical)
\emph{Racah-Wigner $6j$-symbol} is a real number
$\left\{
\begin{matrix}
j_1 \;  j_2 \;  j_3 \\
j_4 \;  j_5 \;  j_6
\end{matrix}
\right\}$
associated
to six
non-negative
half-integer numbers $j_1$, $j_2$, $j_3$, $j_4$, $j_5$, $j_6$ and,
originally, it is implicitly defined (see \cite{CFS})
in terms of the transition matrix between
the following two basis of
$\text{Hom}_{\sl(2)}\big(V(k),V(a)\otimes V(b)\otimes V(c)\big)$
\[
\big\{V(k)\to V(p)\otimes V(c)\to V(a)\otimes V(b)\otimes
V(c)\big\}_{p\ge0}
\]
and
\[
\big\{V(k)\to V(a)\otimes V(q)\to V(a)\otimes V(b)\otimes
V(c)\big\}_{q\ge0}.
\]
In fact, the $6j$-symbol can be defined as above in a more general context,
in particular for any semisimple multitensor category, see for instance
\cite{EFK}.

As far as we know, Theorem \ref{thm.6jIntro} is not known and
gives a new and clean definition of the Racah-Wigner $6j$-symbol
that is explicit, in contrast to the original
implicit definition. Even though there are several formulas
expressing the $6j$-symbol as a sum rational numbers we did not
find in the literature any explicit structural definition of the
$6j$-symbol in terms of the representation theory of $\sl(2)$. 
Our
proof of Theorem \ref{thm.6jIntro} is based on a long and
technical computation (performed in~\S\ref{cal8}), and it remains
elusive to us a proof of it based only on the original definition
of the $6j$-symbol and the representation theory of $\sl(2)$.

In this paper, we use Theorem \ref{thm.6jIntro}
to determine when the Lie algebra $\mathcal{L}$
mentioned above is abelian and we obtain the following result.

\begin{theorem}
Let $a$, $b$, $c$ and $m$ be non-negative integers such that $V(m)$ is
an $\sl(2)$-submodule of both $V(a)\otimes V(b)$ and $V(b)\otimes
V(c)$. Let $\mathcal{J}$ be the image of $\Lambda^2(V(m))$ under
the map
\begin{equation*}
V(m)\otimes V(m)\rightarrow\mathrm{Hom}(V(b),V(a))\otimes
\mathrm{Hom}(V(c),V(b))\to \mathrm{Hom}(V(c),V(a)).
\end{equation*}

Then the following conditions are equivalent:

\begin{enumerate}[(1)]

\item $\mathcal{L}$ is abelian.

\item $\mathcal{J}=0$.

\item $
\left\{\begin{matrix}
\frac{m}{2} \; \frac{k}{2} \; \frac{m}{2} \\[1mm]
\frac{a}{2} \; \frac{b}{2} \; \frac{c}{2}
\end{matrix}
\right\}= 0 $ for all non-negative integers $k$ satisfying
$k\equiv 2m-2 \mod 4 $.

\item Up to a swap of $a$ and $c$ we have: $c=0$, $b=m$, $a\equiv
2m \mod 4$ and $a\leq 2m$; or $b=c+m$ and $a=c+2m$. \item
$\Lambda^2(V(m))$ is disjoint from $\mathrm{Hom}(V(c),V(a))$.

\item There is a uniserial $\sl(2)\ltimes V(m)$-module with socle
factors $V(a),V(b),V(c)$.

\end{enumerate}
\end{theorem}

The proof of this theorem requires the determination of
non-trivial zeros of the $6j$-symbol within certain parameters.
Finding non-trivial zeros of the $6j$-symbol, is in general, a
very difficult problem (see for instance \cite{L}, \cite{R} or
\cite{ZR}) and we think that the above theorem might have
applications to it.

\section{Matrix recognition of uniserial modules}\label{sec.Matrix recognition}

 Let $\g$ be a Lie algebra with solvable radical $\r$ and Levi
decomposition $\g=\s\ltimes \r$, and fix a representation $T:\g\to\gl(V)$. Given
a basis $B$ of $V$ we let $M_B:\g\to\gl(d)$, $d=\dim(V)$, stand for
the corresponding matrix representation.

By an $\s$-basis of $V$ of type 1 we understand a basis of the
form $B=B_1\cup\cdots\cup B_n$, where each $B_i$ is a basis of an
$\s$-submodule $W_i$ of $V$, and
\begin{equation}
\label{tipo} 0\subset W_1\subset W_1\oplus W_2\subset W_1\oplus
W_2\oplus W_3\subset \cdots\subset W_1\oplus\cdots \oplus
W_n=V
\end{equation}
is the socle series of $V$. We likewise define an $\s$-basis of
type 2 by requiring that (\ref{tipo}) be a composition series of
$V$. Since $V$ is completely reducible as an $\s$-module it is
clear that bases of both types exist. In either case $B$ gives
rise to a sequence $V_0,V_1,...,V_n$ of $\g$-modules defined by
$V_0=0$ and $V_i=W_1\oplus\cdots\oplus W_i$ for $1\leq i\leq n$.

\begin{lemma}
\label{b} The ideal $[\g,\r]$ annihilates
every irreducible $\g$-module.
\end{lemma}

\begin{proof} An elementary proof can be found, for instance, in Lemma 2.4
of \cite{CS}.
\end{proof}

\begin{cor}
\label{b1} If $B$ is
any $\s$-basis of $V$ then $M_B(s)$ is block diagonal and $M_B(r)$
is strictly block upper triangular for all $s\in\s$ and $r\in [\g,\r]$.
\end{cor}

\begin{lemma}
\label{super} If $B$ is an $\s$-basis of type 1 then
none of the blocks in the first superdiagonal of $M_B(\r)$ is
identically 0.
\end{lemma}

\begin{proof} Let $B$ and $W_1,...,W_n$, $V_0,V_1,...,V_n$ be as in (\ref{tipo}).
Let $2\leq i\leq n$ and suppose, if possible, that the block
$(i-1,i)$ of $M_B(\r)$ is identically 0. It follows easily that
$W_i\oplus V_{i-2}$ is a submodule of $V$. Clearly $(W_i\oplus
V_{i-2})/V_{i-2}$ and $V_{i-1}/V_{i-2}$ are non-zero submodules of
$V/V_{i-2}$ having trivial intersection. This contradicts the fact
that the socle of $V/V_{i-2}$ is $V_{i-1}/V_{i-2}$.
\end{proof}

\begin{theorem}
\label{fed} The $\g$-module $V$ is uniserial if and only if given
any $\s$-basis $B$ of type~2 none of the blocks in the first
superdiagonal of $M_B(\r)$ is identically 0. If $\g$ is perfect and there exists one $\s$-basis
$B$ of type 2 such that none of the blocks in the first
superdiagonal of $M_B(\r)$ is identically 0 then $V$ is
uniserial.
\end{theorem}

\begin{proof} If $V$ is uniserial an $\s$-basis of type 2 is also of type 1,
so Lemma \ref{super} applies. If $V$ is not uniserial then some
factor of its socle series is not irreducible. This factor is a
completely reducible $\g$-module, which easily yields an
$\s$-basis $B$ of type~2 with at least one block in the first
superdiagonal of $M_B(\r)$ identically 0.

Suppose next $\g$ is perfect and let $B$ be an $\s$-basis of type 2 such that none of the blocks in the first
superdiagonal of $M_B(\r)$ are identically 0. As indicated above,
$B$ gives rise to a series of $\s$-modules $W_1,...,W_n$ and
$\g$-modules $V_0,V_1,...,V_n$ in such a way that (\ref{tipo})
composition series of $V$. We will show that (\ref{tipo}) is in fact the socle series of $V$.

Arguing by induction, it suffices to show that
$\mathrm{soc}(V)=W_1$. Let $U$ be a non-zero submodule of~$V$. We
wish to show that $W_1\subseteq U$. Since $U\cap V\neq 0$ there
exists a smallest index $1\leq j\leq n$ such that $U\cap V_j\neq
0$. If $j=1$ we are done by the irreducibility of $W_1$. Suppose,
if possible, that $1<j\leq n$. The definition of $j$ ensures the
existence of $u\in U$ such that $u=w_1+\cdots+w_j$, where $w_i\in
W_i$ for $1\leq i\leq j$ and $w_j\neq 0$. Let $r\in\r$. Then $r\in
[\g,\r]$, since $\g$ is perfect, so $ru=rw_2+\cdots+rw_j$, where
$rw_i\in V_{i-1}$ for all $2\leq i\leq j$ by Lemma \ref{b}. In
particular $\r u\in U\cap V_{j-1}$. The choice of $j$ forces
forces $\r u=0$, so $\r w_j\subseteq V_{j-2}$. Let $T$ be 0-weight
space for the action of $\r$ on $V_j/V_{j-2}$. As $\r$ is an ideal
of $\g$, the subspace $T$ is $\s$-invariant. Since $w_j+V_{j-2}\in
T$ the $\s$-submodule of $V_j/V_{j-2}$ generated by $w_j+V_{j-2}$
is contained in~$T$, i.e. $(W_j+ V_{j-2})/V_{j-2}\subseteq T$,
which means $\r W_j\subseteq V_{j-2}$, a contradiction.
\end{proof}

\begin{note}
\label{op1}
{\rm Let $\g$ be any imperfect Lie algebra. Then there is a non-uniserial
$\g$-module with an $\s$-basis
$B$ of type 2 such that none of the blocks in the first
superdiagonal of $M_B(\r)$ is identically 0. It suffices to find a counterexample
when $\g=\C x$ is one dimensional,
in which case we can take $x\mapsto \left(%
\begin{array}{cc}
  1 & 1 \\
  0 & 0 \\
\end{array}%
\right)$.}
\end{note}

\begin{lemma}
\label{reves} If $U$ is a subspace of $V$ let $U^0=\{f\in V^*\,|\, f(U)=0\}$.
Then $U\mapsto U^0$ is an inclusion reversing bijective correspondence from the $\g$-submodules
of $V$ to those of $V^*$. Moreover, if $U\subseteq W$ are $\g$-submodules of $V$ then
$U^0/W^0\cong (W/U)^*$ via $f+W^0\mapsto \widetilde{f}$, where $\widetilde{f}(w+U)=f(w)$. In particular,
if $0=V_0\subset V_1\subset\cdots \subset V_n=V$ is a (resp. the only) composition series of $V$, with composition factors $X_1,\dots,X_n$,
then $0=V_n^0\subset\cdots\subset V_1^0\subset V_0^0=V^*$ is a (resp. the only) composition series of $V^*$, with composition factors $X_n^*,\dots,X_1^*$.
\end{lemma}

\begin{proof} Use the natural isomorphism of $\g$-modules $V\to V^{**}$.
\end{proof}

\begin{note} Lemma \ref{b} through Theorem \ref{fed} are
valid, mutatis mutandis, for an arbitrary finite dimensional
complex associative algebra $\mathfrak{A}$. Both $\r$ and
$[\g,\r]$ are to be replaced by the Jacobson radical
$\mathfrak{J}$ of $\mathfrak{A}$, and $\s$ by a semisimple
subalgebra of $\mathfrak{A}$ complementing $\mathfrak{J}$, whose existence is
ensured by the Wedderburn-Malcev theorem.

Lemma \ref{reves} is also valid if $\mathfrak{A}$ has an involution $a\mapsto a^*$
and we make $V^*$ into an $\mathfrak{A}$-module via $(af)(v)=f(a^* v)$.

\end{note}

\section{Admissible sequences}
\label{secext}

In this section $\g=\s\ltimes \r$, where $\s$ and $\r$ are arbitrary Lie
algebras.

Let $W_1,...,W_n$ be $\s$-modules and set $V=W_1\oplus\cdots\oplus
W_n$. Let $T_i:\s\to\gl(W_i)$ and $T:\s\to\gl(V)$ stand for the
associated representations. As is well known, $\gl(V)$ becomes an
$\s$-module via $s\cdot f=[T(s),f]$, for $s\in\s$ and
$f\in\gl(V)$.

Suppose first that $X:\g\to\gl(V)$ is a representation extending $T$ and denote by $Y:\r\to\gl(V)$ the
restriction of $X$ to $\r$. Then
$$
Y(s\cdot r)=Y([s,r])=X([s,r])=[X(s),X(r)]=[T(s),Y(r)]=s\cdot Y(r),
$$
i.e. $Y$ is a homomorphism of $\s$-modules.

Suppose conversely that $Y:\r\to\gl(V)$ is a homomorphism of $\s$-modules and let $X:\g\to\gl(V)$ be the linear
extension of $T$ and $Y$ to $\g$. Then
$$
X([s,r])=Y([s,r])=Y(s\cdot r)=s\cdot Y(r)=[T(s),Y(r)]=[X(s),X(r)],
$$
i.e. $X$ preserves all brackets $[s,r]$. As $T$ is a Lie
homomorphism, $X$ also preserves all brackets $[s_1,s_2]$. But
$Y$, and hence $X$, need not preserve the brackets $[r_1,r_2]$.

In any case, we will identify $\gl(V)$ with $\underset{1\leq i,j\leq n}\bigoplus \Hom(W_i,W_j)$ as $\s$-modules by interpreting
each linear map $W_i\to W_j$ as a linear map $V\to V$ that is 0 on all summands $W_k$ of $V$ with $k\neq i$. Suppose we
are given $n(n-1)/2$ $\s$-homomorphisms $f_{i,j}:\r\to \Hom(W_i,W_j)$, where $i>j$, and let $Y:\r\to\gl(V)$ be the $\s$-homomorphism
corresponding to them. At this point we make the simplifying assumption that $\r$ be abelian. Then $Y([r,t])=0$ for all $r,t\in\r$.
On the other hand, in order to have $[Y(r),Y(t)]=0$ it is necessary that the following maps vanish:
\begin{equation}
\label{nma}
f_{i+1,i}(r)f_{i+2,i+1}(t)-f_{i+1,i}(t)f_{i+2,i+1}(r)\in \Hom(W_{i+2},W_i),\quad 1\leq i\leq n-2.
\end{equation}
Equivalently, the Lie subalgebra of $\gl(V)$ generated by all elements of $\gl(V)$ of the form
$f_{2,1}(r)+\cdots+ f_{n,n-1}(r)$, $r\in\r$, must be abelian.
Obviously this condition is also sufficient if restrict our list of starting maps to $f_{2,1},\dots, f_{n,n-1}$.

Now each of the maps $\r\times \r\to \Hom(W_{i+2},W_i)$ defined by (\ref{nma}) is alternating, thereby
giving rise to a linear map $\Lambda^2\r\to \Hom(W_{i+2},W_i)$. A
simple calculation shows that this is a homomorphism of
$\s$-modules. Thus a sufficient condition for the vanishing of (\ref{nma}) is that  $\Lambda^2\r$ be disjoint with the
$\s$-modules $W_3^*\otimes W_1$,...,$W_n^*\otimes W_{n-2}$.  The
necessity of this condition is examined in \S\ref{sec.length3}.
Taking into account the preceding discussion and Theorem \ref{fed}
the following criterion is established.

\begin{prop}
\label{xrite}

Let $\g=\s\ltimes\r$ be a Levi decomposition and let $W_1,...,W_n$
be a sequence of irreducible $\s$-modules.

(a) Suppose $W_1,...,W_n$ is admissible (as defined in \S\ref{intro}). Then $\r$
is not disjoint with any of the $\s$-modules $W_2^*\otimes
W_1$,...,$W_n^*\otimes W_{n-1}$. In particular, if $\r$ is
irreducible, it must be a constituent of $W_2^*\otimes
W_1$,...,$W_n^*\otimes W_{n-1}$.

(b) Assume $\g$ is perfect and $\r$ is abelian. Then $W_1,...,W_n$ is admissible
if and only if $\r$ is not disjoint with
any of the $\s$-modules $W_2^*\otimes W_1$,...,$W_n^*\otimes
W_{n-1}$ and for some choice of non-zero $\s$-homomorphisms
$f_{i+1,i}:\r\to \Hom(W_{i+1},W_i)$ the Lie subalgebra of $\gl(V)$, $V=W_1\oplus\cdots\oplus
W_n$, generated by $f_{2,1}(r)+\cdots + f_{n,n-1}(r)$, $r\in\r$, is abelian. In particular, if
$\Lambda^2\r$ is disjoint with
the $\s$-modules $W_3^*\otimes W_1$,...,$W_n^*\otimes
W_{n-2}$ and $\r$ is not disjoint  with
any of the $\s$-modules $W_2^*\otimes W_1$,...,$W_n^*\otimes
W_{n-1}$ then $W_1,...,W_n$ is admissible.
\end{prop}

In regards to uniqueness, we have the following criterion.

\begin{prop}
\label{d} Suppose $\g$ is perfect, with Levi decomposition $\g=\s\ltimes\r$. Let $W_1,...,W_n$ be irreducible
$\s$-modules satisfying:

(a) $\dim(\Hom_\s(\r,W_{i+1}^*\otimes W_i))=1$ if $1\leq i<n$.

(b) $\dim(\Hom_\s(\r,W_{j}^*\otimes W_i))=0$ if $j-i\geq 2$.

Then there exists at most one $\g$-module $V$, up to isomorphism,
with socle factors $W_1,...,W_n$.
\end{prop}

\begin{proof} Let $V$ be one such a module. Since its socle
factors are irreducible, $V$ is uniserial. Let $B$ be an
$\s$-basis of $V$. Since $\g$ is perfect, Corollary \ref{b1}
ensures that $M_B(\r)$ is strictly block upper triangular. By
Lemma \ref{super} none of the blocks in the first superdiagonal of
$M_B(\r)$ are identically 0, while (b) guarantees that all other
strictly upper triangular blocks of $M_B(\r)$ are identically 0.
By (a) the blocks in the first superdiagonal of $M_B(\r)$ are
uniquely determined up to a non-zero scalar (which depends only on
the position of the block). Conjugating all $M_B(x)$, $x\in\g$, by
a suitable block diagonal matrix, with each block a scalar matrix,
we can arbitrarily scale all blocks in the first superdiagonal.
This yields the desired result.
\end{proof}

\section{Existence and uniqueness of the uniserial module $Z(\la,b)$}
\label{hist}

 The notation introduced here will be kept for the
remainder of the paper. Let $\s$ be a non-zero semisimple Lie algebra with Cartan
subalgebra~$\h$, associated root system $\Phi$, and fixed system of simple
roots $\Pi$. The coroots $h_\al\in\h$ associated to the simple
roots $\al\in\Pi$ form a basis of $\h$. The basis of $\h^*$ dual
to $\{h_\al\,|\,\al\in\Pi\}$
 consists of the fundamental weights $\{\la_\al\,|\,\al\in\Pi\}$.
Let $\Lambda^+$ stand for the dominant integral weights of $\h$
associated to $\Pi$, i.e. the non-negative integral linear
combinations of the fundamental weights $\la_\al$. Given
$\la,\mu\in\h^*$ we declare $\la\leq\mu$ if $\mu-\la$ is a
non-negative rational linear combination of simple roots. It is
well-known that the inverse of the Cartan matrix has non-negative
rational coefficients. It follows that all fundamental weights are
strictly positive. Therefore, all non-zero dominant integral
weights are strictly positive. This fact will be repeatedly and
implicitly used below.

Let $W$ stand for the Weyl group of $\Phi$ and write $w_0$ for the longest element of~$W$, i.e. the one sending $\Pi$ to $-\Pi$.

We fix $\mu\in\Lambda^+$ and let $V(\mu)$ stand for an irreducible
$\s$-module with highest weight $\mu$. Define the dual weight
$\mu^*=-w_0\mu\in\Lambda^+$, noting that $V(\mu)^*\cong V(\mu^*)$.
We assume henceforth that that $\mu\neq 0$ and consider the
perfect Lie algebra $\g=\s\ltimes \r$, where $\r=V(\mu)$. By the
special case we mean the case $\g=\sl(2)\ltimes V(m)$, $m\geq 1$.

\begin{theorem}\label{thm.general Z(l,b)}
\label{constr} Let $\la\in\Lambda^+$ and $b\geq 0$. Then, up to
isomorphism, there exists one and only one uniserial $\g$-module,
say $Z(\la,b)$, that has length $b+1$ and socle factors
$V(\la),V(\la+\mu^*),...,V(\la+b\mu^*)$.
\end{theorem}

\begin{proof} In light of Proposition \ref{xrite}, existence follows if we can
prove that $V(\mu)$ is not disjoint with any of $V(\la^*+(i+1)\mu)\otimes V(\la+i\mu^*)$
while $\Lambda^2 V(\mu)$ is disjoint with all $V(\la^*+(i+2)\mu)\otimes V(\la+i\mu^*)$. Now
$$
\Hom(V(\mu),V(\la^*+(i+1)\mu)\otimes V(\la+i\mu^*))\cong V(\mu)^*\otimes V(\la^*+(i+1)\mu)\otimes V(\la+i\mu^*),
$$
as $\s$-modules, so
$$
\Hom_\s(V(\mu),V(\la^*+(i+1)\mu)\otimes V(\la+i\mu^*))\cong \Hom_\s(V(\la+(i+1)\mu^*),V(\mu^*)\otimes V(\la+i\mu^*))
$$
as vector spaces. It is clear that the latter space is not only non-zero but
in fact one dimensional.

Reasoning as above and using that fact that
$(\Lambda^2 V(\mu))^*\cong \Lambda^2 V(\mu^*)$
we see that the vector space $\Hom_\s(\Lambda^2 V(\mu),V(\la^*+(i+2)\mu)\otimes V(\la+i\mu^*))$
is isomorphic to $\Hom_\s(V(\la+(i+2)\mu^*),\Lambda^2 V(\mu^*)\otimes V(\la+i\mu^*))$. But
the latter is 0 since all weights of $\Lambda^2 V(\mu^*)\otimes V(\la+i\mu^*)$ are strictly less than $\la+(i+2)\mu^*$.

In regards to uniqueness, note that $V(\la^*+j\mu)\otimes
V(\la+i\mu^*)$ and $V(\mu)$ are disjoint provided $j-i\geq 2$.
This follows as above by observing that all weights of
$V(\mu^*)\otimes V(\la+i\mu^*)$ are strictly less than
$\la+j\mu^*$. Now apply Proposition \ref{d}.
\end{proof}

\begin{cor}
\label{dual} In the notation of Theorem \ref{constr}, there exists
one and only one uniserial $\g$-module, namely $Z(\la,b)^*$, that
has length $b+1$ and whose socle factors are
$V(\la^*+b\mu),...,V(\la^*+\mu),V(\la^*)$.
\end{cor}

\begin{proof} Immediate consequence of Theorem \ref{constr} and
Lemma \ref{reves}.
\end{proof}

\begin{note} It is clear that all modules constructed in Theorem \ref{constr} and Corollary \ref{dual}
are non-isomorphic from each other, except in the obvious case
$b=0$ and $\la=\la^*$, when $Z(\la,b)=V(\la)\cong
V(\la)^*=Z(\la,b)^*$.
\end{note}

\begin{note}
\label{sl2} In the special case $\g=\sl(2)\ltimes V(m)$ these  results read as follows. Given integers $\ell,b\geq 0$
there exists one and only one uniserial $\g$-module having socle
factors either $V(\ell),V(\ell+m),...,V(\ell+bm)$ or
$V(\ell+bm),...,V(\ell+m),V(\ell)$. These modules will be
respectively denoted by $Z(\ell,b)$ and $Z(\ell,b)^*$.
\end{note}

\begin{note}
\label{lendos}
 Suppose $\g$ is perfect with Levi decomposition $\s\ltimes \r$ such that $[\r,\r]=0$.
Let $W_1,W_2$ be irreducible $\s$-modules. By Proposition
\ref{xrite} there exists a uniserial $\g$-module $V$ with socle
factors $W_1,W_2$ if and only if $\r$ is not disjoint with
$W_2^*\otimes W_1$.

Let $H=\Hom_\s(\r,W_2^*\otimes W_1)$ and set $P=P(H)$, the
associated projective space (i.e. the points of $P$ are the the
lines of $H$ through the origin). It is not difficult to see that
the isomorphism classes of such modules $V$ are parametrized by
the points of $P$. In particular, if $\dim(H)>1$ there are
infinitely many such classes.


In the special case $\g=\sl(2)\ltimes V(m)$ the possibility $\dim(H)>1$ never arises.
Indeed, we have $W_1=V(a)$, $W_2=V(b)$, with $0\leq a\leq b$ (otherwise
consider $V^*$). Our previous comments and the Clebsch-Gordan
formula ensure that such $V$ exists if and only if $m$ is in the
list of numbers $b-a,b-a+2,...,b+a-2,b+a$. But then $V$
will be unique since $V(b)\otimes V(a)$ is multiplicity free and
Proposition \ref{d} applies.

\end{note}

\section{An explicit matrix realization of the $\sl(2)\ltimes V(m)$-module $Z(\ell,b)$}
\label{secex}

In this section we consider the special case $\g=\sl(2)\times V(m)$ and construct a matrix realization
of $Z(\ell,b)$, where $\ell\geq 0$ and $b\geq 0$. Taking the opposite transpose of our representation yields a
matrix version of $Z(\ell,b)^*$.

The Lie algebra $\g$ has basis
$e,h,f,v_0,v_1,...,v_m$, subject to the following relations:
\begin{equation}
\label{rel1}
[h,e]=2e,\quad [h,f]=-2f,\quad [e,f]=h,
\end{equation}
\begin{equation}
\label{rel2}
[v_i,v_j]=0,\quad 0\leq i,j\leq m,
\end{equation}
\begin{equation}
\label{rel3}
[h,v_i]=(m-2i)v_i,\quad [e,v_i]=(m-(i-1))v_{i-1},\quad [f,v_i]=(i+1)v_{i+1},
\end{equation}
where $0\leq i\leq m$ and $v_{-1}=0=v_{m+1}$.

For $a\geq 0$ let $H(a),E(a),F(a)$ be the matrices in $\gl(a+1)$ all of whose entries are 0, except that
the diagonal of $H(a)$ is $(a,a-2,\dots,-a+2,-a)$ and, if $a\geq 1$, the first superdiagonal of $E(a)$ is $(a,\dots,2,1)$ and
the first subdiagonal of $F(a)$ is $(1,2,\dots,a)$. Set
$H(\ell,b)=H(\ell)\oplus\cdots\oplus H(\ell+b m)$,
$E(\ell,b)=E(\ell)\oplus\cdots\oplus E(\ell+b m)$ and
$F(\ell,b)=F(\ell)\oplus\cdots\oplus F(\ell+b m)$.


For $a\geq 0$ we define the $m+1$ rectangular matrices $W_0(a),...,W_m(a)$, all of size $(a+1)\times (a+m+1)$,
as follows:
$$
W_0(a)=(0_{(a+1)\times m} I_{a+1}),
W_1(a)=(0_{(a+1)\times (m-1)} I_{a+1} 0_{(a+1)\times 1}),
$$
$$
W_2(a)=(0_{(a+1)\times (m-2)} I_{a+1} 0_{(a+1)\times 2}),...,
W_m(a)=(I_{a+1}0_{(a+1)\times m}).
$$
We next scale these matrices to produce
$$
V_i(a)=(-1)^{i} \binom{m}{i}\; W_i(a),\quad 0\leq i\leq m.
$$

For $0\leq i\leq m$ let $V_i(\ell,b)$ be the block partitioned matrix all of whose blocks are equal to 0, except that if $b\geq 1$
the blocks along the first superdiagonal are $V_i(\ell),V_i(\ell+m),...,V_i(\ell+(b-1)m)$.

\begin{lemma} The map $h\mapsto H(\ell,b)$, $e\mapsto E(\ell,b)$, $f\mapsto F(\ell,b)$,
$v_i\mapsto V_i(\ell,b)$, where $0\leq i\leq m$, defines a matrix representation of $\g$ with associated module $Z(\ell,b)$.
\end{lemma}

\begin{proof} Since $H(\ell,b),E(\ell,b),F(\ell,b)$ satisfy (\ref{rel1}) we see that
the map $h\mapsto H(\ell,b)$, $e\mapsto E(\ell,b)$, $f\mapsto F(\ell,b)$ defines a matrix representation of $\sl(2)$ whose associated module decomposes
as $V(\ell)\oplus V(\ell+m)\oplus\dots\oplus V(\ell+b m)$.

Given $a\geq 0$ let $U_i(a)$  be the matrix partitioned into 4 blocks, whose (1,2) block is
$W_i(a)$ and all other blocks are 0. Direct calculation shows that
$$
[H(a)\oplus H(a+m), U_0(a)]=m U_0(a),\quad [E(a)\oplus E(a+m), U_0(a)]=0
$$
and
\begin{equation}
\label{rel4}
[F(a)\oplus F(a+m),U_i(a)]=-(m-i) U_{i+1}(a),\; 0\leq i\leq m.
\end{equation}
Thus $U_0(a),\dots,U_m(a)$ is a basis for an $\sl(2)$-module, say $S_a$, isomorphic to $V(m)$. From (\ref{rel4}) we get
$$
f^i U_0(a)=(-1)^i m(m-1)\cdots (m-(i-1))U_i(a)=(-1)^i m!/(m-i)! U_i(a),\;\; 0\leq i\leq m,
$$
It follows that
$$
f^i U_0(a)/i!=(-1)^{i} \binom{m}{i}\; U_i(a),\quad 0\leq i\leq m
$$
is a basis for $S_a$ upon which $h,e,f$ act as in (\ref{rel3}). Hence $V_0(\ell,b),\dots,V_m(\ell,b)$ is a basis of an $\sl(2)$-module
upon which $h,e,f$ act via $H(\ell,b),E(\ell,b),F(\ell,b)$ as in~(\ref{rel3}).

Next we verify that the relations (\ref{rel2}) are preserved. By means of the actions of $e$ and $f$ on $V_0(\ell,b),\dots,V_m(\ell,b)$, it suffices to
verify $[V_0(\ell,b),V_m(\ell,b)]=0$. This easily reduces to the case $b=2$, which is confirmed through a simple calculation.

We thus have a matrix representation of $\g$. By Theorem \ref{fed} the associated module, say $V$, is uniserial with socle factors $V(\ell),V(\ell+m),...,V(\ell+b m)$, so $V\cong Z(\ell,b)$ by Proposition \ref{d}.
\end{proof}

Here we present a matrix realization for $m=2$ and $V\cong Z(1,2)$.
{\scriptsize
\[
\begin{array}{rr|rrrr|rrrrrr}
  h & e  & v_2 & -2v_1 &  v_0 &  0 \\
  f & -h & 0  & v_2 & -2v_1 & v_0  \\
 \hline
       & &  3h & 3e &  0 &  0  & v_2 & -2v_1 &  v_0 &  0 &  0 &  0 \\
       & &   f  & h   & 2e & 0 &  0  &  v_2 & -2v_1 & v_0 &  0 &  0    \\
       & &   0  & 2f  & -h & e  &  0  &  0  &  v_2 & -2v_1 & v_0 &  0     \\
       & &   0  &  0  & 3f  & -3h  &  0  &  0  &  0  &  v_2 & -2v_1 & v_0  \\
 \hline
       & &  & &  & &  5h  & 5e  &  0  &  0  &  0 &  0     \\
       & &  & &  & &   f  & 3h  & 4e  &  0  &  0 &  0       \\
       & &  & &  & &   0  & 2f  & h   & 3e  &  0  &  0       \\
       & &  & &  & &   0  &  0  & 3f  & -h  &  2e &  0      \\
       & &  & &  & &   0  &  0  &  0  & 4f  & -3h  &  e  \\
       & &  & &  & &   0  &  0  &  0  &  0  &  5f  &  -5h  \\
\end{array}
\]
}
A matrix realization of $Z(1,2)^*$ is obtained by taking the opposite transpose of the
above matrix. A suitable change of basis presents a realization of $Z(1,2)^*$
as the following block upper triangular matrices:
{\scriptsize
\[
\begin{array}{rrrrrr|rrrr|rr}
  5h  & 5e  &  0  &  0  &  0 &  0   &  10v_0  &  0    &  0  &  0  \\
  f  & 3h  & 4e  &  0  &  0 &  0    &  4 v_1  &  6v_0 &  0    &  0     \\
  0  & 2f  & h   & 3e  &  0  &  0   &    v_2  &  6v_1  & 3v_0    &  0      \\
  0  &  0  & 3f  & -h  &  2e &  0   &   0    &   3v_2  &  6v_1    &  v_0     \\
  0  &  0  &  0  & 4f  & -3h  &  e  &   0    &   0     &   6v_2   &  4v_1     \\
  0  &  0  &  0  &  0  &  5f  &  -5h &   0    &   0     &   0     &  10v_2  \\
 \hline
      & &   & &  & &   3h & 3e  &  0  &  0   & 3v_0 &   0    \\
      & &   & &  & &   f  & h   & 2e  &   0  & 2v_1 &   v_0     \\
      & &   & &  & &   0  & 2f  & -h  &  e   &  v_2 &  2v_1      \\
      & &   & &  & &   0  &  0  & 3f  & -3h  &  0   &  3v_2  \\
 \hline
     & &    & &   & &   & &  & &    h & e   \\
      & &   & &   & &   & &  & &    f & -h   \\
\end{array}
\]
}

\section{Characterization of the uniserial modules $Z(\la,b)$ and  $Z(\la,b)^*$}
\label{ty}

We adhere to the notation introduced in \S \ref{hist}. Let $V$ be a $\g$-module. By a weight vector we mean a non-zero
common eigenvector for the action of $\h$ on $V$. A highest weight vector, or just a maximal vector,
is a weight vector that is annihilated by all $e_\al$, $\al\in\Pi$.
The weight spaces $\r_{\mu}$ and $\r_{w_0\mu}$ as well as the root spaces $\s_\al$, $\al\in\Pi$,
are all one dimensional and we fix a spanning vector for each of them, say $u_{\mu}\in\r_{\mu}$, $u_{w_0\mu}\in\r_{w_0\mu}$
and $e_\al\in\s_\al$.

\begin{lemma}
\label{ayq} Let $V$ be a $\g$-module and let $v\in V$.

(a) Let $\al\in\Pi$. Then $e_\al u_{w_0\mu}^i v=0$ for all $i\geq 0$ if and only if $e_\al v=0$ and $[e_\al, u_{w_0\mu}]v=0$.

(b) Let $i\geq 0$. If $e_\al u_{w_0\mu}^i v=0$ for all $\al\in \Pi$ then $\r u_{w_0\mu}^i v$
is included in the $\s$-submodule of $V$ generated by $u_{w_0\mu}^{i+1} v$.

(c)  Let $\al\in\Pi$. If $e_\al v=0$ then  $e_\al u_\mu^i v=0$ for all $i\geq 0$.
\end{lemma}

\begin{proof} (a) Necessity is obvious. As for sufficiency, we argue by induction. The base case $i=0$ is given. Suppose the result is true for some $i\geq 0$.
Since $\r$ is abelian
$$
e_\al u_{w_0\mu}^{i+1}v=u_{w_0\mu} e_\al u_{w_0\mu}^i v+[e_\al, u_{w_0\mu}]u_{w_0\mu}^i v=0+u_{w_0\mu}^i [e_\al, u_{w_0\mu}]v=0.
$$

(b) It suffices to prove this for elements of $\r$ of the form $[e_{\al_1},\dots,[e_{\al_s},u_{w_0 \mu}]\dots]$ as these span $\r$.
If $s=0$ the multibracket reduces to $u_{w_0\mu}$ and
the result is true by definition. This case and the stated hypothesis yield the result by induction.

(c) Since $[e_\al,u_\mu]=0$ we obtain $e_\al u_\mu^i v=u_\mu^i e_\al v=0$.
\end{proof}

\begin{theorem}
\label{recog} Let $V$ be a $\g$-module,
$\la\in\Lambda^+$, $b\geq 0$ and let $u=u_{w_0\mu}$  (resp. $u=u_{\mu}$).
Then $V\cong Z(\la,b)$ (resp. $V\cong Z(\la,b)^*$)
if and only if there is a
vector $v\in V$ satisfying conditions (C1),(C2),(C3) (resp. (C1),(C2),(C3)$^*$) stated below.
\begin{enumerate}
  \item[(C1)] $v$ is a maximal vector of weight $\la+b\mu^*$ (resp. $\lambda^*$)
that generates $V$ as a $\g$-module.
  \item[(C2)] $u^b v\neq 0$, $u^{b+1}v=0$.
   \item[(C3)] $[e_\al,u_{w_0\mu}]v=0$ for all $\al\in\Pi$.
   \item[(C3)$^*$] $\r u_\mu^iv$ is included in the $\s$-module generated by $u_\mu^{i+1}v$ for all $0\leq i\leq b$.
\end{enumerate}
Moreover, in such case $v$ is unique up to scaling.
\end{theorem}

\begin{proof} This naturally breaks into two parts.

\medskip

\noindent {\sc Sufficiency.} Let $W_i$ be the $\s$-submodule of $V$ generated by $u^i v$
for $i=0,\dots,b$. Then (C1)-(C3) and Lemma \ref{ayq} ensure that
$u^i v$ is a maximal vector of weight $\la+(b-i)\mu^*$ (resp. $\lambda^*+i\mu$), so
$W_{i}\cong V(\la+(b-i)\mu^*)$ (resp. $W_{i}\cong V(\la^*+i\mu)$).
Let $V_{b+1}=0$ and set
$$
V_{b-i}=W_b\oplus\cdots \oplus W_{b-i},\quad 0\leq i\leq b.
$$
We claim that $V_{b-i}$ is a $\g$-submodule of $V$ for all $i=0,\dots,b$.
It suffices to show that $\r V_{b-i}\subset V_{b-(i-1)}$. Since the 0-weight space
of $\r$ acting on any $\g$-module is $\s$-invariant and
$V_{b-i}/V_{b-(i-1)}$ is $\s$-irreducible, it suffices to prove
that $\r$ has a common 0-eigenvector in $V_{b-i}/V_{b-(i-1)}$.
We contend that $u^{b-i}v+V_{b-(i-1)}\in V_{b-i}/V_{b-(i-1)}$ is non-zero
and annihilated by $\r$. That $u^{b-i}v\not\in V_{b-(i-1)}$ follows from the
fact that $u^{b-i}v$ is a maximal vector and
none of the maximal vectors of $V_{b-(i-1)}$ has the same weight as $u^{b-i}v$.
It follows from (C2) and Lemma \ref{ayq} (resp. (C2) and (C3)$^*$) that
$\r u^{b-i}v\in V_{b-(i-1)}$. This proves our contention and hence the claim.

The $\g$-invariance of $V_0$ and (C1) now yield $V=V_0$. We have shown that
$$
0=V_{b+1}\subset V_b\subset\cdots\subset V_0=V
$$
is a composition series of the $\g$-module $V$, with composition factors
$$
V_{b-i}/V_{b-(i-1)}\cong V(\la+i\mu^*)\quad \text{(resp. $V_{b-i}/V_{b-(i-1)}\cong V(\la^*+(b-i)\mu)$)}.
$$
Since $u u^{b-i}v\in W_{b-(i-1)}$, $0<i\leq b$, Theorem \ref{fed}
ensures that $V$ is uniserial with socle factors
$V(\la),V(\la+\mu^*),...,V(\la+b\mu^*)$
(resp. $V(\la^*+b\mu),...,V(\la^*+\mu),V(\la^*)$).
From the uniqueness part of Theorem \ref{constr} we conclude that $V\cong Z(\la,b)$ (resp. $V\cong Z(\la,b)^*$)

\medskip

\noindent{\sc Necessity.} By assumption the socle series of $V$, say
$$0=V_{b+1}\subset V_b\subset\cdots\subset V_0=V,$$
has irreducible factors $$V_{b-i}/V_{b-(i-1)}\cong V(\la+i\mu^*)\quad (\text{resp.}\; V_{b-i}/V_{b-(i-1)}\cong V(\la^*+(b-i)\mu).$$
Up to scaling $V$ has a unique maximal vector, say $v_i$,
of weight $\la+ i\mu^*$ (resp. $\la^*+(b-i)\mu)$). In any uniserial module, a vector belonging only to the last term of the socle series
generates the entire module. Hence $V$ is generated by $v=v_b\notin V_1$.

We know from Lemma \ref{b} that $\r V_{b-i}\subseteq V_{b-(i-1)}$. In particular $u^{b+1}v=0$.

Suppose next $V\cong Z(\la,b)$. If $0\leq i\leq b$ and $u^i v\neq 0$ then $u^i v$ is a maximal vector
since its weight, namely $\la+(b-i)\mu^*$, is the highest in $V_i$. Since $v\neq 0$ and $u^{b+1} v=0$
there is an index $i$ satisfying $0\leq i\leq b$, $u^i v\neq 0$ and $u^{i+1}v=0$. Our preceding comment
implies $e_\al u^i v=0$ for all $\al\in\Pi$. This and $uu^i v=0$ yield $\r u^i v=0$, that is,
$u^i v\in\mathrm{soc}(V)\cong V(\la)$. But the highest weight in $V(\la)$
is $\la$ and $u^i v$ has weight $\la+(b-i)\mu^*$, so $i=b$. Thus $u^i v$ is a maximal vector for all $0\leq i\leq b$.

Suppose finally $V\cong Z(\la,b)^*$. Thus $V$ has socle factors $V(\la^*+b\mu),\dots, V(\la^*)$
and we view these as $\s$-submodules of $V$. We claim that $u$ sends a maximal vector of $V(\la^*+i\mu)$
into one of $V(\la^*+(i+1)\mu)$ for all $0\leq i<b$. By assumption  $\r V(\la^*+i\mu)$ is non-zero
and included in  $V(\la^*+(i+1)\mu)$. Since $\r$ is an irreducible $\s$-module, it follows that $u_\mu V(\la^*+i\mu)\neq 0$,
so there is a weight vector $w$ in $V(\la^*+i\mu)$ not annihilated by $u_\mu$. Since $[e_\al,u_\mu]=0$ for all $\al\in\Pi$
by repeatedly applying the $e_\al$ to $u_\mu w$ we may assume that $u_\mu w$ is a maximal vector of $V(\la^*+(i+1)\mu)$,
in which case $w$ must have weight $\la^*+i\mu$. This proves our claim. Since $v$ is a maximal vector of $V(\la^*)$
we deduce that $u^i v$ is a maximal vector of $V(\la^*+i\mu)$ for all $0\leq i\leq b$. Moreover, $\r u^b v\subseteq \r V(\la^*+b\mu)=0$
and $\r u^i v\subseteq \r V(\la^*+i\mu)\subseteq V(\la^*+(i+1)\mu)$, which is $\s$-generated by $u^{i+1} v$ for all $0\leq i<b$.
\end{proof}

\section{A natural construction of the $\s\ltimes V(\mu)$-module $Z(0,b)$}
\label{sec5}

As mentioned in the Introduction, there have been recent constructions of indecomposable modules for a Lie algebra $\g$ by
embedding $\g$ into a semisimple Lie algebra $\t$ and restricting an irreducible $\t$-module
to $\g$. In this section we use the characterization given in Theorem \ref{recog} to
produce $Z(0,b)$ in the spirit just described.


We adopt the notation introduced at the beginning of
\S\ref{hist} and \S\ref{ty}. In particular, $\g=\s\ltimes \r$, where
$\r=V(\mu)$. Consider the $\s$-module $W=\r^*\oplus\C w$, where
$\s$ acts trivially on  $\C w$. We make $W$ into a $\g$-module as
follows:
\begin{equation}
\label{formi} (s+r)(f+aw)=sf+f(r)w,\quad s\in\s, r\in\r,f\in
\r^*,a\in\C.
\end{equation}
This gives a representation $\g\to\sl(W)$ (which is an embedding
if $\r$ is faithful and, in particular, if $\s$ is simple).

Let $f\in\r^*$ be the only linear functional such that
$f(u_{w_0\mu})=1$ and $f(z)=0$ for any $z$ belonging to a weight
space in $\r$ of weight different from $w_0\mu$. It is clear from
this definition that
$$f\in(\r^*)_{-w_0\mu}=(\r^*)_{\mu^*}.$$
Fix $b\geq 0$ and let $X=S^b(W)$, the $b$th symmetric power of
$W$. This is an irreducible $\sl(W)$-module. We view $X$ as a
$\g$-module via the Lie homomorphism $\g\to\sl(W)$. Let $V$ be the
$\g$-submodule of $X$ generated by $f^b$.

\begin{theorem} The $\g$-module $V$ is isomorphic to $Z(0,b)$. Moreover, the $\g$-module $X$ is
indecomposable, with trivial socle, full socle series
\begin{equation}
\label{tu} 0=\mathrm{soc}^0(X)\subset \mathrm{soc}^1(X)\subset
\mathrm{soc}^2(X)\subset\cdots\subset \mathrm{soc}^{b+1}(X)=X,
\end{equation}
and socle factors
$$\mathrm{soc}^{i+1}(X)/\mathrm{soc}^i(X)\cong S^i(\r^*)\cong S^i(V(\mu^*)),\quad 0\leq i\leq b.$$
Thus the socle factors of $V$, namely
$V(0),V(\mu^*),...,V(b\mu^*)$, are precisely the top summands of
the socle factors of~$X$. In particular, $X$ itself need not be
uniserial, but it is so in the very special case $\g=\sl(2)\ltimes
V(1)$, when $X=V$.
\end{theorem}

\begin{proof} The first assertion follows at once from Theorem
\ref{recog} applied to $v=f^b$. Indeed, $v$ is clearly a maximal vector of $V$ of weight $b\mu^*$
that generates $V$ as a $\g$-module. Moreover, $$u^b f^b=b!
w^b\neq 0\text{ and }u^{b+1}f^b=0,\quad u=u_{w_0\mu}$$ and the very definition of $f$ gives
$$[e_\al,u]f=f([e_\al,u])w=0,\text{ so } [e_\al,u]f^b=0,\quad
\al\in\Pi,\; u=u_{w_0\mu}.$$

As remarked in the Introduction,
$\mathrm{soc}^{i+1}(X)/\mathrm{soc}^i(X)$ is the 0-weight space
for the action of $\r$ on $X/\mathrm{soc}^i(X)$. The formula
(\ref{formi}) makes it clear that
$$\mathrm{soc}^{i+1}(X)=w^b S^0(\r^*)\oplus  w^{b-1}S^1(\r^*)\oplus\cdots\oplus  w^{b-i}S^i(\r^*),\quad 0\leq i\leq b
$$
which gives the isomorphisms of $\s$-modules
$$
\mathrm{soc}^{i+1}(X)/\mathrm{soc}^i(X)\cong w^{b-i}S^i(\r^*)\cong
S^i(\r^*),\quad 0\leq i\leq b.
$$
The remaining assertions now follow immediately.
\end{proof}

The special case $\g=\sl(2)\ltimes V(m)$ can be translated as
follows. Let $\g\to \gl(m+2)$ be the matrix representation defined
in \S\ref{secex} for the uniserial $\g$-module with socle factors $V(0),V(m)$.

Let $S$ be the algebra of polynomials in $m+2$ variables
$X_1,...,X_{m+2}$. This is a module for $\gl(m+2)$, where each
basic matrix $E_{ij}$ acts via derivations on $S$ by means of
$M_{X_i}\circ \partial/\partial X_j$, i.e. partial differentiation
with respect to $X_j$ followed by multiplication by $X_i$.

Given $b\geq 0$, the subspace $X$ of $S$ of all homogeneous
polynomials of degree $b$ is $\gl(m+2)$-stable. We may thus view
$X$ as a $\g$-module via $\g\to \gl(m+2)$, and consider the
$\g$-submodule $V$ of $X$ generated by $X_2^b$. It follows immediately
from Theorem \ref{recog} that $V\cong Z(0,b)$.

Note finally that in the very special case $m=1$ we have $X=V$. In this case, by factoring the terms of
the socle series of $V$ we obtain all $\g$-modules
$Z(\ell,b)$, $\ell\geq 0$.

\section{Other uniserial modules}\label{sec.Other Uniserials}

The uniserial modules $Z(\la,b)$ and their duals are not the only possible ones, even in
the special case $\g=\sl(2)\times V(m)$. We already noted this when dealing with uniserial
modules of length in Note \ref{lendos}. In this section we produce further exceptions, in this case of lengths 3 and 4.
 It is be shown in \S\ref{sec.Classification} that no other exceptions exist.

We maintain throughout the notation introduced
in \S\ref{hist}. In particular, $\g=\s\ltimes \r$, where
$\r=V(\mu)$.

\begin{lemma}
\label{poq}
 Let $\la\in\Lambda^+$. There is a unique uniserial $\g$-module with socle
factors $V(0),V(\mu^*),V(\la)$ provided $V(\la)$ occurs once
$V(\mu^*)\otimes V(\mu^*)$ but not in $\Lambda^2(V(\mu^*))$.
Equivalently, when $V(\la)$ occurs once in $S^2(V(\mu^*))$ and
$V(\mu^*)\otimes V(\mu^*)$.
\end{lemma}

\begin{proof} This follows easily from Propositions
\ref{xrite} and \ref{d}, except for uniqueness when $\la=\mu^*$.
In this case there is a uniserial $\g$-module $V$ with socle
factors $V(0),V(\mu^*),V(\mu^*)$ and we need to establish the
uniqueness of $V$ up to isomorphism.

Let $B$ be an $\s$-basis of $V$ which yields identical matrix
representations of $\s$ on $W_2,W_3$ in the notation of
(\ref{tipo}). Our hypotheses ensure that each of the blocks (1,2),
(2,3), (1,3) of $M_B(\r)$ is completely determined, up to a scalar, once
$B$ is fixed. Moreover, this scalar must be non-zero in the first
two cases. The (1,3) block of $M_B(\r)$ is a scalar multiple, say
by $a\in\C$, of the (1,2) block. Conjugating by the block matrix
$$
\left(%
\begin{array}{ccc}
  1 & 0 & 0 \\
  0 & 1 & a \\
  0 & 0 & 1 \\
\end{array}%
\right)
$$
we obtain a new matrix representation which is identical to the
first except that the (1,3) block of $M_B(\r)$ is now surely 0. The result
follows.
\end{proof}

Under the hypotheses of Lemma \ref{poq}, there is
 a unique uniserial module with socle factors $V(\la^*),V(\mu),V(0)$, dual to the above. Clearly, there is at most one uniserial module of both types, namely the self-dual module with socle factors $V(0),V(\mu),V(0)$,
where $\mu=\mu^*$. We next find explicit conditions for the existence of such a module. Let $\mu=\mu^*$. Then, up to scaling, there is one and only one non-zero $\s$-invariant bilinear form~$\phi:V(\mu)\times V(\mu)\to\C$,
necessarily non-degenerate. By our discussion in \S\ref{secext} there is a uniserial $\g$-module with
socle factors $V(0),V(\mu),V(0)$ if and only if the $\s$-homomorphism $F:\Lambda^2 V(\mu)\to\mathrm{Hom}(\C,\C)$ associated to (\ref{nma})
is trivial. But $F_{v\wedge w}(a)=a(\phi(v,w)-\phi(w,v))$
for all $a\in\C$ and $v,w\in V(\mu)$. This is 0 if and only if $\phi$ is symmetric. We have proven

\begin{lemma}
\label{nodeg1} If $\mu=\mu^*$ there is a uniserial module
with socle factors $V(0),V(\mu),V(0)$ if and only if the non-zero $\s$-invariant bilinear form on $V(\mu)$ is symmetric.
\end{lemma}

\begin{note}
\label{ocho3}
 Suppose that $\mu=\mu^*$ and the
hypotheses of Lemma \ref{poq} are met for $\la=\mu$, i.e., $V(\mu)$ occurs once in $S^2(V(\mu))$ and
$V(\mu)\otimes V(\mu)$. Then, clearly,
there is a uniserial module with socle factors $V(0),V(\mu),V(\mu),V(0)$. However,
in this case the
isomorphism classes of such modules are parametrized by the
complex numbers. Indeed, once all diagonal blocks of $M_B(\s)$ as
well as the first superdiagonal blocks $M_B(\r)$ have been fixed
and the block (1,3) of $M_B(\r)$ has been cleared, there is no way
to modify the block (2,4) of $M_B(\r)$.
\end{note}

We next adapt the above observations to the special case $\g=\sl(2)\ltimes V(m)$. Recall that $V(m)\otimes V(m)=V(2m)\oplus V(2m-2)\oplus\cdots\oplus V(2)\oplus V(0)$, where $S^2(V(m))=V(2m)\oplus V(2m-4)\oplus\cdots$ and $\Lambda^2(V(m))=V(2m-2)\oplus V(2m-6)\oplus\cdots$.
Thus $V(\ell)$ appears in $S^2(V(m))$ if and only if $\ell\leq 2m$ and $2m\equiv \ell\mod 4$. In particular,
$V(0)$ appears in $S^2(V(m))$ (that is, the non-zero $\sl(2)$-invariant bilinear form on $V(m)$ is symmetric) if and only if $m$ is even. We have shown

\begin{lemma} Let $\ell\geq 0$.
Then there is a unique uniserial module with socle factors $V(0),V(m),V(\ell)$ if $\ell\leq 2m$ and $2m\equiv \ell\mod 4$,
in which case the dual module is uniserial with socle factors $V(\ell),V(m),V(0)$.
Moreover, there is a unique uniserial module with socle factors $V(0),V(m),V(0)$ if and only if $m$ is even. Furthermore,
if $m\equiv 0\mod 4$ there is a parametrization by $\C$
of the isomorphism classes of uniserial modules with the same
socle factors $V(0),V(m),V(m),V(0)$.
\end{lemma}

\begin{example}\label{ex.V(0)V(3)V(2)}
A matrix realization of a uniserial $\sl(2)\ltimes V(3)$-module with socle factors $V(0)$, $V(3)$, $V(2)$ is:
{\scriptsize
\[
\begin{array}{r|rrrr|rrr}
  0 & -v_3 & 3v_2  & -3v_1 & v_0 &   \\
 \hline
        &  3h  & 3e  &  0 &   0  & -3v_1 &  3v_0 &  0    \\
        &   f  & h   & 2e &   0  & -2v_2 &   v_1 & v_0    \\
        &   0  & 2f  & -h &   e  & - v_3 &  -v_2 & 2v_1   \\
        &   0  &  0  & 3f & -3h  &  0    & -3v_3 & 3v_2  \\
 \hline
        &      &     &    &      &  2h   & 2e    &  0    \\
        &      &     &    &      &   f   &  0    &  e        \\
        &      &     &    &      &   0   & 2f    & -2h      \\
\end{array}.
\]
}

\noindent
 The one parameter family,
parametrized by  $z\in\C$,
of non-isomorphic uniserial
$\sl(2)\ltimes V(4)$-modules with socle factors $V(0)$, $V(4)$, $V(4)$, $V(0)$ is given by:
{\scriptsize
\[
\begin{array}{r|rrrrr|rrrrr|r}
0 & v_4 & -4v_3 & 6v_2 & -4v_1 & v_0  &   &   &   &   &   &   \\
 \hline
  & 4h & 4e & 0 & 0 & 0 & 6v_2 & -12v_1 & 6v_0  & 0 & 0 & z\,v_0  \\
  & f & 2h & 3e & 0 & 0 & 3v_3 & -3v_2 & -3v_1 & 3v_0  & 0 & z\,v_1 \\
  & 0 & 2f & 0 & 2e & 0 & v_4 & 2v_3 & -6v_2 & 2v_1 & v_0  & z\,v_2 \\
  & 0 & 0 & 3f & -2h & e & 0 & 3v_4 & -3v_3 & -3v_2 & 3v_1 & z\,v_3 \\
  & 0 & 0 & 0 & 4f & -4h & 0 & 0 & 6v_4 & -12v_3 & 6v_2 & z\,v_4 \\
 \hline
  &   &   &   &   &   & 4h & 4e & 0 & 0 & 0 & v_0  \\
  &   &   &   &   &   & f & 2h & 3e & 0 & 0 & v_1 \\
  &   &   &   &   &   & 0 & 2f & 0 & 2e & 0 & v_2 \\
  &   &   &   &   &   & 0 & 0 & 3f & -2h & e & v_3 \\
  &   &   &   &   &   & 0 & 0 & 0 & 4f & -4h & v_4 \\
 \hline
  &   &   &   &   &   &   &   &   &   &   &  0
\end{array}.
\]
}
\end{example}

Let $\g=\sl(3)\ltimes \C^3$ and let
$\lambda_1$ and $\lambda_2$ be the fundamental weights of $\sl(3)$.
We now show that there exists a unique uniserial $\g$-module
with socle factors
\[
V(2\lambda_2),\;V(\lambda_1+\lambda_2),\;V(2\lambda_1).
\]
Notice that in contrast to the other examples considered so far, none of
the differences between the highest weights of these three $\sl(3)$-modules
is a dominant weight.

Let $\r=\C^3$. According to Propositions \ref{xrite} and \ref{d}, it suffices to prove
that
\[
\Hom_{\sl(3)}(\r,V(\lambda_1+\lambda_2)^*\otimes V(2\lambda_2) ) \ne 0,\;
\Hom_{\sl(3)}(\r,V(2\lambda_1)^*\otimes V(\lambda_1+\lambda_2) ) \ne 0,
\]
and
\[
\Hom_{\sl(3)}(\r,V(2\lambda_1)^*\otimes V(2\lambda_2)) =0=
\Hom_{\sl(3)}(\Lambda^2\r,V(2\lambda_1)^*\otimes V(2\lambda_2)).
\]
Since $\lambda_1^*=\lambda_2$ and
\[
 \r\cong V(\lambda_1),\quad\Lambda^2\r\cong V(\lambda_2),
\]
the above conditions follow from the following tensor product decompositions:
\begin{align*}
 V(\lambda_1+\lambda_2)^*\otimes V(2\lambda_2)
 & =
 V(\lambda_1+\lambda_2)  \otimes V(2\lambda_2)  \\
 & =
 V(\lambda_1+3\lambda_2)\oplus  V(2\lambda_1+\lambda_2)\oplus  V(2\lambda_2) \oplus  V(\lambda_1), \\
 V(2\lambda_1)^*\otimes V(2\lambda_2)
 & =
 V(2\lambda_2)  \otimes V(2\lambda_2) \\
 & =
 V(4\lambda_2)\oplus  V(\lambda_1+2\lambda_2)\oplus  V(2\lambda_1).
\end{align*}

Here is an example with $\s=\so(m)$, $m\geq 3$, and $\mu=\la_1$, the first fundamental weight.
A matrix representation of $\so(m)\ltimes V(\la_1)$ that
is uniserial with socle factors $V(0),V(\la_1),V(0)$ can be obtained as
follows.

Let $U$ be a vector space of
dimension $m$, and let $f:U\times U\to\C$ be a non-degenerate
symmetric bilinear form. The subalgebra of $\gl(U)$ preserving $f$
is $\so(m)$, and $U$ is the natural module for $\so(m)$.

Set $n=m+2$ and let $J$ be the $n\times n$ matrix with 1's along the secondary
diagonal and 0's elsewhere. The $n\times n$ matrices $X$
satisfying $X'J+JX=0$, where $X'$ indicates the transpose of $X$, form $\so(n)$. The appearance of
such $X$ is
$$
X=\left(
    \begin{array}{ccccc}
      a & x_1 & \cdots & x_m & 0 \\
      y_1 & * & * & * & -x_m \\
      \vdots &  \vdots &  \vdots &  \vdots & \vdots \\
      y_m & * & * & * & -x_1 \\
      0 & -y_m & \cdots & -y_1 & -a\\
    \end{array}
  \right)
$$
where the inner $*$ follow the same pattern as the outer entries,
i.e. $X$ is skew-symmetric relative to the secondary diagonal. The
subalgebra of $\so(n)$ formed by all $X$ having 0's in the outer
rows/columns is clearly isomorphic to $\so(m)$. Moreover, the
subspace of all $X$ having 0's in the first column, last row, and all inner entries, is
normalized by $\so(m)$, in such a way that together they form a subalgebra
isomorphic to $\so(m)\ltimes U$. Theorem \ref{fed} ensures that this is a uniserial
representation, whose socle factors are clearly $V(0),U,V(0)$.

In particular when $m=3$ we obtain a uniserial module for
$\sl(2)\ltimes V(2)$ with socle factors $V(0),V(2),V(0)$.
Explicitly, we have the  embedding of $\sl(2)\ltimes V(2)$ into
$\so(5)$
$$\left(
  \begin{array}{ccccc}
    0 & a & b & c & 0 \\
    0 & h & e & 0 & -c \\
    0 & f & 0 & -e & -b \\
    0 & 0 & -f & -h & -a \\
    0 & 0 & 0 & 0 & 0 \\
  \end{array}
\right),$$ which makes $\C^5$ into a uniserial module with socle factors $V(0),V(2),V(0)$.

Thus the isomorphism $\sl(2)\cong\so(3)$ yields a uniserial module for $\sl(2)\ltimes V(2)$ with socle factors $V(0),V(2),V(0)$, where $V(2)$ is the natural module for $\so(3)$.
But when we identify $\sl(2)$ with $\sy(2)$ the invariant form on $V(1)$ is skew-symmetric and
no uniserial module for $\sl(2)\ltimes V(1)$ with socle factors $V(0),V(1),V(0)$ exists, as indicated above.

It is perhaps worth noting that when we pass to a perfect Lie algebra whose radical is nilpotent of class 2,
in addition to all modules arising from the abelian case, we may obtain some new ones as well. As an illustration, let $m\geq 1$,
let $U$ be a vector space of dimension $2m$ and let $f:U\times U\to\C$
be a non-degenerate skew-symmetric bilinear form. The subalgebra of $\gl(U)$ preserving $f$
is $\sy(2m)$ and $U$ is the natural module for $\sy(2m)$. The Heisenberg algebra $\h(2m+1)$ can be defined on
the vector space $U\oplus \C$ by declaring $[u+a,v+b]=f(u,v)$. Then $\sy(2m)$ acts via derivations
on $\h(2m+1)$ by $[x,u+a]=x(u)$, and we may then form the perfect Lie algebra $\sy(2m)\ltimes \h(2m+1)$.
We have a natural Lie epimorphism $\sy(2m)\ltimes \h(2m+1)\to \sy(2m)\ltimes U$, which allows us to view every uniserial module
for $\sy(2m)\ltimes U$ as one for $\sy(2m)\ltimes \h(2m+1)$ in which the center acts trivially.
We wish to construct a uniserial module $V$ for $\sy(2m)\ltimes \h(2m+1)$ with socle factors $V(0),U,V(0)$.
Since the $\sy(2m)$-invariant form on $U$, namely $f$, is skew-symmetric, our earlier comments ensure
that no such module exists for $\sy(2m)\ltimes U$. The reader will also note that the center of
$\sy(2m)\ltimes \h(2m+1)$ does not act trivially on $V$.

Set $n=m+1$ and let $J$ be the $n\times n$ matrix with 1's along the secondary diagonal and 0's everywhere else.
Let $S$ be the $2n\times 2n$ skew-symmetric invertible matrix
$$S=\left(
                     \begin{array}{cc}
                       0 & J \\
                       -J & 0 \\
                     \end{array}
                   \right).
$$
The $2n\times 2n$ matrices $X$ satisfying $X'S+SX=0$ form $\sy(2n)$.
The appearance of such $X$ is
$$
X=\left(
    \begin{array}{cccccccc}
      a &  x_1 & \cdots &  x_m &  y_1 & \cdots & y_m & z \\
      b_1 & * & * & * & * & * & * & y_m \\
      \vdots & * & * & * & * & * & * & \vdots \\
      b_m & * & * & * & * & * & * & y_1 \\
      c_1 & * & * & * & * & * & * & -x_m \\
      \vdots & * & * & * & * & * & * & \vdots \\
      c_m & * & * & * & * & * & * & -x_1 \\
      d & c_m & \cdots & c_1 & -b_m & \cdots & -b_1 & -a \\
    \end{array}
  \right),
$$
where the inner $*$ follow the same pattern as the outer entries. More explicitly,
if we partition $X$ into 4 blocks of size $n\times n$, that is, $X=\left(\begin{array}{cc}
                       A & B\\
                       C & D \\
                     \end{array}
                   \right)$, then $B,C$ are symmetric relative to the secondary diagonal and $D$
                   is the opposite of the transpose of $A$ relative to the secondary diagonal.
In particular, the inner $*$ form
a subalgebra isomorphic to $\sy(2m)$, which together with the first row, and last column,  with $a=0$,
form a subalgebra isomorphic to $\sy(2m)\ltimes \h(2m+1)$. This makes the column space $V=\C^{2n}$ into a uniserial module
for $\sy(2m)\ltimes \h(2m+1)$ with socle factors $V(0),U,V(0)$.

In particular when $m=1$ we get a uniserial module for $\sl(2)\ltimes \h(3)$ whose socle factors are
$V(0),V(1),V(0)$. Explicitly, this is obtained through the following embedding of $\sl(2)\ltimes \h(3)$ into $\sy(4)$:
$$
\left(
  \begin{array}{cccc}
    0 & x & y & z \\
    0 & a & b & y \\
    0 & c & -a & -x \\
    0 & 0 & 0 & 0 \\
  \end{array}
\right).
$$

\section{Admissible sequences of length 3 for $\g=\sl(2)\ltimes V(m)$ }
\label{sec.length3}

Let $\g$ be an arbitrary Lie algebra and let $X,Y,Z$ be
$\g$-modules. Then the map $\mathrm{Hom}(Y,X)\otimes
\mathrm{Hom}(Z,Y)\to \mathrm{Hom}(Z,X)$, defined by
$\alpha\otimes\beta\mapsto\alpha\beta$,
is an epimorphism $\g$-modules.

Let $M,N$ be $\g$-modules and let $f:M\to \mathrm{Hom}(Y,X)$, $g:N\to \mathrm{Hom}(Z,Y)$
be homomorphisms of $\g$-modules. They give rise to the homomorphism of $\g$-modules
\begin{equation}
M\otimes N\stackrel{f\otimes g}{\longrightarrow}\mathrm{Hom}(Y,X)\otimes \mathrm{Hom}(Z,Y)\to \mathrm{Hom}(Z,X).
\label{delone}
\end{equation}
We are interested in the image, say $\mathcal{I}$, of this map. Here is a matrix interpretation.
Let $d_X,d_Y,d_Z$ be the dimensions of $X,Y,Z$ and fix bases $B_X,B_Y,B_Z$ on them. Let $M_X:\g\to \gl(d_X)$,
$M_Y:\g\to \gl(d_Y)$, $M_Z:\g\to \gl(d_Z)$ be the matrix representations associated to the modules $X,Y,Z$
relative to the bases $B_X,B_Y,B_Z$. Consider the $\g$-module $U=X\oplus Y\oplus Z$ of dimension $d=d_X+d_Y+d_Z$
and basis $B=B_X\cup B_Y\cup B_Z$. Then the matrix representation $M:\g\mapsto \gl(d)$ associated to $U$ relative to $B$ is
$$
M(t)=\left(
  \begin{array}{ccc}
    M_X(t) & 0 & 0 \\
    0 & M_Y(t) & 0 \\
    0 & 0 & M_Z(t) \\
  \end{array}
\right),\quad t\in\g.
$$
We view $\gl(d)$ as a $\g$-module by means of $t\cdot A=[M(t),A]$
for $t\in\g$ and $A\in\gl(d)$. Let $M_{a,b}$ stand for the space of all complex matrices of size $a\times b$. Then
$$
\left\{\left(\begin{array}{ccc}
    0 & A & 0 \\
    0 & 0 & 0 \\
    0 & 0 & 0 \\
  \end{array}
\right): A\in M_{d_X,d_Y}\right\}\text{ and
}\left\{\left(\begin{array}{ccc}
    0 & 0 & 0 \\
    0 & 0 & A \\
    0 & 0 & 0 \\
  \end{array}
\right): A\in M_{d_Y,d_Z}\right\}
$$
are $\g$-submodules of $\gl(d)$, respectively isomorphic to $\mathrm{Hom}(Y,X)$ and $\mathrm{Hom}(Z,Y)$. Let
$$
\mathcal{D}=\left\{\left(\begin{array}{ccc}
    0 & f(m) & 0 \\
    0 & 0 & g(n) \\
    0 & 0 & 0 \\
  \end{array}
\right): m\in M, n\in N\right\},
$$
where, by abuse of notation, $f(m)$ and $g(n)$ stand for their own matrices relative to the bases $B_Y, B_X$ and $B_Z,B_Y$.
Then $\mathcal{I}$ is, relative to the bases $B_Z,B_X$, the subspace generated by all matrices
$$
\left\{\left(\begin{array}{ccc}
    0 & 0 & f(m)g(n) \\
    0 & 0 & 0 \\
    0 & 0 & 0 \\
  \end{array}
\right): m\in M, n\in N\right\}.
$$
Note that $\mathcal{A}=\mathcal{D}\oplus \mathcal{I}$
is the associative algebra generated by $\mathcal{D}$.

We are also interested in the case $M=N$. In this case we have the $\g$-module decomposition
$M\otimes M=\Lambda^2(M)\oplus S^2(M)$. Let $\mathcal{J}$ be the the image of $\Lambda^2(M)$ under (\ref{delone}).
Then $\mathcal{J}$ is, relative to the bases $B_Z,B_X$, the subspace generated by all matrices
$$
\left\{\left(\begin{array}{ccc}
    0 & 0 & f(m)g(n)-f(n)g(m) \\
    0 & 0 & 0 \\
    0 & 0 & 0 \\
  \end{array}
\right): m,n\in M\right\}
$$
and, in this case, $\mathcal{L}=\mathcal{D}_{\text{diag}}\oplus \mathcal{J}$
is the Lie algebra generated by
$$
\mathcal{D}_{\text{diag}}=\left\{\left(\begin{array}{ccc}
    0 & f(m) & 0 \\
    0 & 0 & g(m) \\
    0 & 0 & 0 \\
  \end{array}
\right): m\in M\right\}.
$$

Let $a,b,c,p,q$ be non-negative integers. We next focus attention on the case:
$$\g=\sl(2),\; M=V(p),\; N=V(q),\; X=V(a),\; Y=V(b),\; Z=V(c).
$$
We wish to determine the $\g$-module structure of $\mathcal{I}$ (resp. $\mathcal{J}$ when $p=q$).
Since the tensor product of irreducible $\sl(2)$-modules is multiplicity free, $\mathcal{I}$
is independent of the choice of $f$ and $g$ (as long as they are non-zero)
and determining $\mathcal{I}$ is equivalent to finding all $k\geq 0$ such that
$V(k)$ is a submodule of~$\mathcal{I}$.


\begin{theorem}
\label{imp2} Let $a,b,c,p,q,k$ be non-negative integers. Assume the existence of $\sl(2)$-embeddings
$V(p)\to\mathrm{Hom}(V(b),V(a))$ and $V(q)\to\mathrm{Hom}(V(c),V(b))$, and let $\mathcal{I}$ be the image of the corresponding
$\sl(2)$-homomorphism
\begin{equation}
\label{jjota}
V(p)\otimes V(q)\rightarrow\mathrm{Hom}(V(b),V(a))\otimes \mathrm{Hom}(V(c),V(b))\to \mathrm{Hom}(V(c),V(a)).
\end{equation}
Then $V(k)$ appears in $\mathcal{I}$ if and only if the Wigner-Racah $6j$-symbol (see \S\ref{Subsec.Clebsch-Gordan})
\begin{equation}
\label{conjot}
\left\{\begin{matrix}
\frac{q}{2} \; \frac{k}{2} \; \frac{p}{2} \\[1mm]
\frac{a}{2} \; \frac{b}{2} \; \frac{c}{2}
\end{matrix}
\right\}\neq 0.
\end{equation}
\end{theorem}

\begin{proof} Let $V(p)\to V(a)\otimes V(b)$ and  $V(q)\to V(b)\otimes V(c)$ be $\sl(2)$-embeddings,
and let $j:V(b)\to V(b)^*$ be an $\sl(2)$-isomorphism. Let $\mathcal{K}$ be the image of the corresponding $\sl(2)$-homomorphism
$$
V(p)\otimes V(q)\rightarrow  V(a)\otimes V(b)\otimes  V(b)\otimes V(c)\stackrel{\mu}{\rightarrow}  V(a)\otimes V(c),
$$
where
\begin{equation}
\label{estr1}
\mu(x\otimes y_1\otimes y_2\otimes z)=(j(y_1))(y_2)x\otimes z,
\quad x\in V(a),\; y_1,y_2\in V(b),\;  z\in V(c).
\end{equation}
It is not difficult to see that $\mathcal{K}$ is independent of the choice of $j$ and the given embeddings, and that $\mathcal{I}\cong \mathcal{K}$.
Hence, it suffices to prove the result for $\mathcal{K}$.

If $V(k)$ does not occur in $V(p)\otimes V(q)$ or  $V(a)\otimes
V(c)$ then $V(k)$ does not occur in $\mathcal{K}$ and, moreover,
the left hand side of (\ref{conjot}) is 0. Thus, we may assume
that $V(k)$ appears in $V(p)\otimes V(q)$ and $V(a)\otimes V(c)$.
In \S\ref{cal8} we furnish a concrete embedding
$\iota_r^{s,t}:V(r)\to V(s)\otimes V(t)$ for any non-negative
integers such that $|t-s|\leq r\leq t+s$ and $t+s\equiv r\mod 2$,
as well as a fixed isomorphism $j_r:V(r)\to V(r)^*$ for any $r\geq
0$. These data  yield a specific $\sl(2)$-homomorphism
$\phi:V(k)\to V(a)\otimes V(c)$
\begin{equation}
\label{estr2}
 V(k)\overset{\iota_k^{p,q}}{\longrightarrow}V(p)\otimes V(q)\overset{\iota_p^{a,b}\otimes\iota_q^{b,c}}
 {\longrightarrow}V(a)\otimes V(b)\otimes V(b)\otimes V(c)\overset{\mu}{\longrightarrow} V(a)\otimes V(c).
\end{equation}
But, up to scaling, the only $\sl(2)$-homomorphism $V(k)\to V(a)\otimes V(c)$ is $\iota_k^{a,c}$. Hence
$$
\phi=\lambda\iota_k^{a,c}
$$
for a unique scalar $\lambda$. A long and technical calculation (performed independently in \S\ref{cal8}) shows that
$$
\lambda=C\;
\left\{\begin{matrix}
\frac{q}{2} \; \frac{k}{2} \; \frac{p}{2} \\[1mm]
\frac{a}{2} \; \frac{b}{2} \; \frac{c}{2}
\end{matrix}
\right\},
$$
where $C$ is a non-zero scalar explicitly defined in \S\ref{cal8}. The result now follows.
\end{proof}

\begin{cor}
\label{cvu}
 Keep the hypothesis of Theorem \ref{imp2} and suppose that $p=q$.  Let $\mathcal{J}$
 be the image of $\Lambda^2(V(p))$ under (\ref{jjota}).
Then $V(k)$ appears in $\mathcal{J}$ if and only if
$k \equiv 2p-2\mod 4$ and (\ref{conjot}) holds.
\end{cor}

\begin{proof} Clearly $V(k)$ appears in $\mathcal{J}$ if and only if $V(k)$ appears in $\mathcal{I}$ and $\Lambda^2(V(p))$. Recalling that
$V(k)$ appears in $\Lambda^2(V(p))$ if and only if $0\leq k\leq 2p$ and $2p-2\equiv k\mod 4$, the result follows from Theorem \ref{imp2}.
\end{proof}

\begin{example}\label{ex.6j=0}
Let us consider the $\sl(2)$-homomorphism
\[
V(4)\otimes V(4)\rightarrow\mathrm{Hom}(V(6),V(4))\otimes \mathrm{Hom}(V(4),V(6))\to \mathrm{Hom}(V(4),V(4)),
\]
i.e. $p=q=a=c=4$ and $b=6$.
It turns out that
\[
\begin{matrix}
 \left\{\begin{matrix}
\frac{4}{2} \; \frac{0}{2} \; \frac{4}{2} \\[1mm]
\frac{4}{2} \; \frac{6}{2} \; \frac{4}{2}
\end{matrix}
\right\}\!=\!-\frac15,&
\left\{\begin{matrix}
\frac{4}{2} \; \frac{2}{2} \; \frac{4}{2} \\[1mm]
\frac{4}{2} \; \frac{6}{2} \; \frac{4}{2}
\end{matrix}
\right\}\!=\!0,&
\left\{\begin{matrix}
\frac{4}{2} \; \frac{4}{2} \; \frac{4}{2} \\[1mm]
\frac{4}{2} \; \frac{6}{2} \; \frac{4}{2}
\end{matrix}
\right\}\!=\!\frac{4}{35}, &
\left\{\begin{matrix}
\frac{4}{2} \; \frac{6}{2} \; \frac{4}{2} \\[1mm]
\frac{4}{2} \; \frac{6}{2} \; \frac{4}{2}
\end{matrix}
\right\}\!=\!\frac{1}{14},&
\left\{\begin{matrix}
\frac{4}{2} \; \frac{8}{2} \; \frac{4}{2} \\[1mm]
\frac{4}{2} \; \frac{6}{2} \; \frac{4}{2}
\end{matrix}
\right\}\!=\!\frac{1}{70}
\end{matrix}.
\]
This shows that $\mathcal{I}=V(0)\oplus V(4)\oplus V(6)\oplus V(8)$ and thus
$\mathcal{J}=V(6)$.
Similarly, it can be shown that for $\sl(2)$-homomorphism
\[
V(4)\otimes V(4)\rightarrow\mathrm{Hom}(V(2),V(4))\otimes \mathrm{Hom}(V(4),V(2))\to
\mathrm{Hom}(V(4),V(4)),
\]
we have
$\mathcal{I}=V(0)\oplus V(2)\oplus V(4)\oplus V(8)$ and thus
$\mathcal{J}=V(2)$.
\end{example}

\begin{definition}\label{def.triangle}
  Given three non-negative integers $a$, $b$ and $c$, we  will say that the triple $(a,b,c)$
  satisfies the \emph{triangle condition} if $a$, $b$ and $c$ are the side lengths
of a (possibly degenerate) triangle and $a+b+c$ is even.
\end{definition}

From the Clebsch-Gordan formula for the decomposition of the tensor product
of two $\sl(2)$-modules, we know that
$V(k)$ is a submodule of $V(a)\otimes V(b)$
if and only if
$|a-b|\le k \le a+b$ and $k\equiv a+b\mod 2$.
It is clear that this is the same as saying that
$(a,b,k)$ satisfies the triangle condition.

In terms of Theorem \ref{imp2}, it is clear that a necessary condition for $V(k)$ to appear in the image $\mathcal{I}$ of (\ref{jjota}) is that the four triples
\begin{equation}\label{eq.triangle}
(k,p,q),\quad (p,a,b),\quad (q,b,c),\quad (k,a,c)
\end{equation}
satisfy the triangle condition. These four triangle conditions can be depicted by the following labeled tetrahedron:
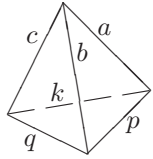
\setlength{\unitlength}{10pt}
\begin{center}
\begin{picture}(10,7)(0,-2)
 \put(0,0){\line(1,2){2.1}}
 \put(0,0){\line(2,-1){3}}
\put(3,-1.5){\line(1,1){2.4}}
\put(3.05,-1.5){\line(-1,6){.95}}
 \put(2.1,4.2){\line(1,-1){3.3}}
 \multiput(0,0)(1.45,.24){4}{\line(6,1){1}}
 \put(3.4,3.0){$a$}
 \put(0.7,2.7){$c$}
 \put(2.6,2.0){$b$}
 \put(4.5,-0.5){$p$}
 \put(0.6,-1.2){$q$}
 \put(1.6,0.5){$k$}
\end{picture}
\end{center}
We point out, however, that the above four triangle conditions do not
imply the existence of an euclidean metric tetrahedron
with side lengths $a$, $b$, $c$, $p$, $q$ and $k$ (as indicated in the above picture);
it is known that an additional condition on the Cayley-Menger determinant is required for that
(see, \cite{Bl}, \cite{GV} or \cite{WD}).

Note also that the four triangle conditions \eqref{eq.triangle}
are not sufficient for $V(k)$ to appear in the image of
(\ref{jjota}), as shown in Example \ref{ex.6j=0}. According to
Theorem~\ref{imp2}, $V(k)$ will not appear in the image of
(\ref{jjota}) if and only if $ \left\{\begin{matrix}
\frac{q}{2} \; \frac{k}{2} \; \frac{p}{2} \\[1mm]
\frac{a}{2} \; \frac{b}{2} \; \frac{c}{2}
\end{matrix}
\right\}=0.
$
We recall that finding the non-trivial zeros of the $6j$-symbol
is a well studied and very difficult problem (see, for instance,
\cite{L}, \cite{R} and the references therein, or more recently \cite{ZR}).
In particular it is known that
\[
 \left\{\begin{matrix}
a \; a-1 \; a \\[1mm]
a \; a+1 \; 2
\end{matrix}
\right\}
=0\quad\text{ and }\quad
 \left\{\begin{matrix}
\frac{j}{2} & \frac{2j-2}{2} &\frac{j}{2} \\[1mm]
\frac{3j-8}{2} & \frac{2j-6}{2} & \frac{j}{2}
\end{matrix}
\right\}
=0
\]
for all integers $a\ge2$ and $j\ge4$ (see equations (4.14) and
(4.15) in \cite{L}).

\medskip

We can now state the following theorem which is important for the classification of the uniserial modules of
the Lie algebra $\sl(2)\ltimes V(m)$.

\begin{theorem}\label{thm:contition}
Let $a$, $b$, $c$ and $m$ be non-negative integers such that $(a,b,m)$
and $(b,c,m)$ satisfy the triangle condition, and
let $\mathcal{J}$ be the image of $\Lambda^2(V(m))$ under
(\ref{jjota}) when $p=m=q$.

Let $f:V(m)\to\mathrm{Hom}(V(b),V(a))$ and
$g:V(m)\to\mathrm{Hom}(V(c),V(b))$ be non-zero
$\sl(2)$-homomorphisms and  consider the Lie
subalgebra, say $\mathcal{L}$, of $\gl(V)$, where $V=V(a)\oplus
V(b)\oplus V(c)$, generated by $f(r)+g(r)$, $r\in V(m)$.

Then the following conditions are equivalent:

\begin{enumerate}[(1)]
\item $\mathcal{L}$ is abelian.
\item $\mathcal{J}=0$.
\item
$
\left\{\begin{matrix}
\frac{m}{2} \; \frac{k}{2} \; \frac{m}{2} \\[1mm]
\frac{a}{2} \; \frac{b}{2} \; \frac{c}{2}
\end{matrix}
\right\}= 0
$ for all non-negative integers $k$ satisfying $k\equiv 2m-2 \mod 4 $.

\item Up to a swap of $a$ and $c$ we have: $c=0$, $b=m$,
$a\equiv 2m \mod 4$ and $a\leq 2m$; or $b=c+m$ and $a=c+2m$.
\item $\Lambda^2(V(m))$ is disjoint from $\mathrm{Hom}(V(c),V(a))$.
\item There is a uniserial $\sl(2)\ltimes V(m)$-module with socle factors $V(a),V(b),V(c)$.
\end{enumerate}
\end{theorem}

\begin{proof} (1)$\Rightarrow$(2) We already noted that
$\mathcal{L}=\mathcal{D}_{\text{diag}}\oplus \mathcal{J}$
is generated by $\mathcal{D}_{\text{diag}}$, so $\mathcal{J}=0$
if $\mathcal{L}$ is abelian.

(2)$\Rightarrow$(3)
If  $0\leq k\leq 2m$ and
$(a,c,k)$ satisfies the triangle condition then
(2) and Corollary \ref{cvu} imply that
$
\left\{\begin{matrix}
\frac{m}{2} \; \frac{k}{2} \; \frac{m}{2} \\[1mm]
\frac{a}{2} \; \frac{b}{2} \; \frac{c}{2}
\end{matrix}
\right\}= 0
$.
For all other $k\equiv 2m-2 \mod 4$, this $6j$-symbol is zero by definition.

(3)$\Rightarrow$(4) As the $6j$-symbol is invariant under the
permutation of any two columns (\cite{CFS}), we may assume that $a\geq c$.

Since $(a,b,m)$ and $(b,c,m)$ satisfy the triangle
condition it follows that $m\ge a-b$, $m\ge b-c$ and thus $2m\ge
a-c$.

We claim that
\begin{equation}\label{eq.equiv4}
  a-c\equiv 2m \mod 4.
\end{equation}
Otherwise, $k=a-c\equiv 2m-2 \mod 4 $ and, since $(a,c,k)$ is a degenerate triangle,
it follows from Property (iii) in the list of properties of
the $6j$-symbols in \S\ref{Subsec.Clebsch-Gordan}, that
$\left\{\begin{matrix}
\frac{m}{2} \; \frac{k}{2} \; \frac{m}{2} \\[1mm]
\frac{a}{2} \; \frac{b}{2} \; \frac{c}{2}
\end{matrix}
\right\}\ne0$
which contradicts (3).

If $a-c=2m$, the triangle conditions on $(a,b,m)$ and $(b,c,m)$
imply that $b=m+c$ and $a=c+2m$.
Otherwise, $a-c<2m$ and we will prove that $c=0$ and thus $b=m$,
$a\equiv 2m \mod 4$ and $a\leq 2m$.
Since $a-c<2m$, it follows from \eqref{eq.equiv4}
that $a-c\le 2m-4$. If $c\ge1$ then
\[
(h,m,m),\quad (h,a,c),\quad (m,a,b),\quad (m,b,c)
\]
satisfy the triangle condition for $h=a-c$ and $h=a-c+2$ and we
are in a position to apply Lemma \ref{lemma.non-zero} to
\[
  j_1=\frac{a-c}2,\quad j_2=j_3=\frac{m}2,\quad j_4=\frac{b}2,\quad
 j_5=\frac{c}2,\quad  j_6=\frac{a}2.
\]
We obtain that
$\left\{\begin{matrix}
\frac{m}{2} \; \frac{k}{2} \; \frac{m}{2} \\[1mm]
\frac{a}{2} \; \frac{b}{2} \; \frac{c}{2}
\end{matrix}
\right\}
=
\left\{\begin{matrix}
\frac{k}{2} \; \frac{m}{2} \; \frac{m}{2}  \\[1mm]
\frac{b}{2} \; \; \frac{c}{2} \; \; \frac{a}{2}
\end{matrix}
\right\}\ne0$
for either $k=a-c+2$ or $k=a-c+6$.
Since both $k$ satisfy $k\equiv 2m-2 \mod 4 $,
this contradicts (3).

(4)$\Rightarrow$(5) Immediate from the decompositions of
$\Lambda^2(V(m))$ and $V(a)\otimes V(c)$.

(5)$\Rightarrow$(1)  This is obvious.

(6) $\!\!\Leftarrow$ $\!\!\!\!\Rightarrow$(1) If $m\geq 1$ this
follows from Proposition \ref{xrite}. If $m=0$ then $a=b=c$ and it
is easy to see that conditions (1) and (6) are both true.
\end{proof}

\section{Classification of uniserial $\sl(2)\ltimes V(m)$-modules}
\label{sec.Classification}

We are finally in a position to show that, except for a few
exceptions of  lengths 2, 3 and 4, every uniserial
$\sl(2)\ltimes V(m)$-module is isomorphic to $Z(\ell,b)$ or its dual.

\begin{theorem}
Let $\g=\sl(2)\ltimes V(m)$, where $m\geq 1$.
Then, up to a reversing of the order, the following are the only
admissible sequences for $\g$:

\noindent
\begin{tabular}{ll}
Length 1. & $V(a)$. \\[2mm]
Length 2. & $V(a),V(b)$, where $a+b\equiv m\mod 2$ and $0\le b-a\leq m\leq a+b$. \\[2mm]
Length 3. & $V(a),V(a+m),V(a+2m)$; or \\[1mm]
      & $V(0),V(m),V(c)$, where $c\equiv 2m \mod 4$ and $c\leq 2m$. \\[2mm]
Length 4. & $V(a),V(a+m),V(a+2m),V(a+3m)$; or \\[1mm]
      & $V(0),V(m),V(m),V(0)$, where $m\equiv 0\mod 4$. \\[2mm]
Length $\geq 5$. &  $V(a),V(a+m),\dots,V(a+s m)$, where $s\geq 4$. \\[2mm]
\end{tabular}

\noindent
Moreover, each of these sequences arises from only one
isomorphism class of uniserial $\g$-modules, except
for the sequence $V(0),V(m),V(m),V(0)$, $m\equiv 0\mod 4$. The isomorphism classes of uniserial $\g$-modules associated to
this sequence are parametrized by the complex numbers, as described in Note \ref{ocho3}.
\end{theorem}

\begin{proof}  That the stated sequences are admissible is proven in \S\ref{hist} and \S\ref{sec.Other Uniserials},
while the uniqueness, up to isomorphism, of the uniserial modules arising from such sequences follows from Proposition \ref{d},
except for the sequence $V(0),V(m),V(m),V(0)$, where $m\equiv 0\mod 4$, which is handled in Note \ref{ocho3}.

It remains to prove that, up to a reversing of the order, the only
admissible sequences are as indicated. Those of length 2 are
considered in Note \ref{lendos}.  Let $V(a),V(b),V(c)$ be an
admissible sequence of length 3, which is condition (1) of Theorem
\ref{thm:contition}. This is equivalent to condition (4) of
Theorem \ref{thm:contition}, so all admissible sequences of length
3 are as stated. Next let $V(a_1),\dots,V(a_n)$ be an admissible
sequence of length $n\geq 4$. By Lemma \ref{reves} we may assume
that $a_1\leq a_n$. Since any submodule or quotient of a uniserial
module is also uniserial, we see that
$V(a_{i-1}),V(a_i),V(a_{i+1})$ is also admissible for any $1<i<n$.
Applying this fact in combination with our determination of all
admissible sequences of length 3, we deduce the following:
$a_i\neq 0$ for $1<i<n$; either $a_1,\dots,a_n$ is strictly
increasing, in which case it does so by a fixed increment of $m$,
or else $n=4$ and $a_1=0,a_2=m,a_3=m,a_4=0$.
\end{proof}


\section{A new interpretation of the Wigner-Racah $6j$-symbol and
the calculation of the scalar $\lambda$}
\label{cal8}

Let $k$ be a non-negative integer and let $e_k$ denote a  highest weight vector
of the irreducible $\sl(2)$-module $V(k)$ of highest weight $k$.
If $\{H,E,F\}$ is the standard basis of
 $\sl(2)$ then
$\mathcal{B}_k=\{F^re_k:r=0,\dots,k\}$ is a basis of  $V(k)$ and
\begin{align*}
HF^re_k &= (k-2r)\,F^re_k,\\
EF^r e_k &= r(k+1-r)\,F^{r-1} e_k.
\end{align*}

We know that $V(k)$ is isomorphic to the dual $\sl(2)$-module $V(k)^*$. In fact,
if $\{(F^r e_k)^*:r=0,\dots,k\}$ is the dual basis of $\mathcal{B}_k$, then
$(F^k e_k)^*$ is the highest weight vector
of  $V(k)^*$ and the map $j_k:V(k)\to V(k)^*$, given by
\begin{equation}
\label{ktodual}
F^r e_k\mapsto(-1)^r(F^{k-r} e_k)^*
\end{equation}
is, up to a scalar, the unique $\sl(2)$-module isomorphism
between $V(k)$ and $V(k)^*$. Suppose $(a,b,k)$ satisfies the triangle condition.
In this case $V(k)$ occurs with multiplicity
one in $V(a)\otimes V(b)$ and
if
\[
x_k^{a,b}=\frac{a+b-k}2
\]
then
\begin{equation}\label{eq:TensorDominant}
v^{a,b}_k=\sum_{r=0}^{x_k^{a,b}} (-1)^r
\frac{\binom{x_k^{a,b}}{r}}{\binom{x_k^{a,b}+k}{a-r}}\; F^{r} e_a\otimes F^{x_k^{a,b}-r} e_b
\end{equation}
is, up to a scalar, the unique highest weight vector of weight $k$
in $V(a)\otimes V(b)$.
We denote by
\[
\iota_k^{a,b}\in\text{Hom}_{\sl(2)}\big(V(k),V(a)\otimes V(b)\big)
\]
the unique $\sl(2)$-module homomorphism sending $e_k$ to $v^{a,b}_k$.

Let us assume that the four triples
\[
(k,p,q),\quad (p,a,b),\quad (q,b,c),\quad (k,a,c)
\]
satisfy the triangle condition. Let $\phi$ be the map defined by (\ref{estr2}), that is
\[
 \phi=\mu\circ\big(\iota^{a,b}_{p}\otimes\iota^{b,c}_{q}\big)\circ\iota^{p,q}_{k},
\]
where $\mu:V(a)\otimes V(b)\otimes V(b)\otimes V(c)$ is defined in (\ref{estr1}) by means of the isomorphism $j=j_b:V(b)\to V(b)^*$ given in (\ref{ktodual}). Explicitly, we have
$$
\mu(x\otimes F^r e_b\otimes F^{b-s} e_b\otimes z)=(-1)^r \delta_{r,s} x\otimes z,\quad x\in V(a), z\in V(c), 0\leq r,s\leq b.
$$
As noted at the end of the proof of Theorem \ref{imp2}, we have $\phi=\lambda\iota_k^{a,c}$. Thus
\[
 \phi(e_k)=\lambda\, v^{a,c}_k,
\]
where $v^{a,c}_k$ is defined in  (\ref{eq:TensorDominant}). We will now compute $\lambda$.
First we have
\[
\iota^{p,q}_{k}e_{k}
=\sum_{r_1=0}^{x^{p,q}_{k}}
(-1)^{r_1}\frac{\binom{x^{p,q}_{k}}{r_1}}{\binom{x^{p,q}_{k}+k}{p-r_1}}\;
F^{r_1} e_p
\otimes
                            F^{x^{p,q}_{k}-r_1} e_q.
\]
\begin{align*}
v_p^{a,b}=\iota^{a,b}_{p}e_{p}
&=
\sum_{r_2=0}^{x^{a,b}_{p}} (-1)^{r_2}\frac{\binom{x^{a,b}_{p}}{r_2}}{\binom{x^{a,b}_{p}+p}{a-r_2}}\;
                F^{r_2} e_a\otimes F^{x^{a,b}_{p}-r_2} e_b,\\
v_q^{b,c}=\iota^{b,c}_{q}e_{q}&=
\sum_{r_3=0}^{x^{b,c}_{q}} (-1)^{r_3}\frac{\binom{x^{b,c}_{q}}{r_3}}{\binom{x^{b,c}_{q}+q}{b-r_3}}\;
F^{r_3} e_b\otimes F^{x^{b,c}_{q}-r_3} e_c.
\end{align*}
Thus
\begin{align*}
\phi(e_{k}) =\mu \sum_{r_1=0}^{x^{p,q}_{k}}
(-1)^{r_1}\frac{\binom{x^{p,q}_{k}}{r_1}}{\binom{x^{p,q}_{k}+k}{p-r_1}}\;
F^{r_1}.&\Big(
\sum_{r_2=0}^{x^{a,b}_{p}} (-1)^{r_2}\frac{\binom{x^{a,b}_{p}}{r_2}}{\binom{x^{a,b}_{p}+p}{a-r_2}}\;
F^{r_2} e_a\otimes F^{x^{a,b}_{p}-r_2} e_b\Big) \\
\otimes
F^{x^{p,q}_{k}-r_1}.&\Big(
\sum_{r_3=0}^{x^{b,c}_{q}} (-1)^{r_3}\frac{\binom{x^{b,c}_{q}}{r_3}}{\binom{x^{b,c}_{q}+q}{b-r_3}}\;
F^{r_3} e_b\otimes F^{x^{b,c}_{q}-r_3} e_c\Big).
\end{align*}
Applying Leibniz rule to compute the action of $F^{r_1}$ and $F^{x^{p,q}_{k}-r_1}$ on tensors we get
\begin{align*}
\phi(e_{k}) =
&
\sum_{}
(-1)^{r_1+r_2+r_3}
\frac{\binom{r_1}{r_4}\binom{x^{p,q}_{k}-r_1}{r_5}\binom{x^{p,q}_{k}}{r_1}\binom{x^{a,b}_{p}}{r_2}\binom{x^{b,c}_{q}}{r_3}}
{\binom{x^{p,q}_{k}+k}{p-r_1}\binom{x^{a,b}_{p}+p}{a-r_2}\binom{x^{b,c}_{q}+q}{b-r_3}}\\
&
\times F^{r_2+r_4} e_a \otimes \mu\Big(F^{x^{a,b}_{p}-r_2+r_1-r_4} e_b
\otimes
F^{r_3+r_5} e_b\Big)\otimes F^{x^{b,c}_{q}-r_3+x^{p,q}_{k}-r_1-r_5} e_c.
\end{align*}
Here the sum runs over all $(r_1,r_2,r_3,r_4,r_5)$ allowed by the binomial coefficients in the numerator.
In what follows, all the sums will run over the indicated indices with the restriction that
the factorial numbers involved are non-negative.
In order to compute $\mu$ we need to consider the case when
\[
x^{a,b}_{p}-r_2+r_1-r_4 + r_3+r_5= b
\]
and thus
\begin{align*}
\phi(e_{k}) =
(-1)^{x^{a,b}_{p}}
\sum_{r_1,r_2,r_3,r_4}
(-1)^{r_4+r_3}
&
\frac{\binom{r_1}{r_4}\binom{x^{p,q}_{k}-r_1}{b-x^{a,b}_{p}-r_1+r_2-r_3+r_4}
\binom{x^{p,q}_{k}}{r_1}\binom{x^{a,b}_{p}}{r_2}\binom{x^{b,c}_{q}}{r_3}}
{\binom{x^{p,q}_{k}+k}{p-r_1}\binom{x^{a,b}_{p}+p}{a-r_2}\binom{x^{b,c}_{q}+q}{b-r_3}}\, \\[2mm]
&\hspace{2.4cm}\times F^{r_2+r_4} e_a\otimes F^{x^{a,c}_{k}-r_2-r_4} e_c.
\end{align*}
We know that $\phi(e_k)=\lambda v_k^{a,c}$. Comparing coefficients at $e_a\otimes F^{x_k^{a,c}}e_c$ yields
\[
(-1)^{x^{a,b}_{p}}
\sum_{r_1,r_3}
(-1)^{r_3}
\frac{\binom{x^{p,q}_{k}-r_1}{b-x^{a,b}_{p}-r_1-r_3}
\binom{x^{p,q}_{k}}{r_1}\binom{x^{b,c}_{q}}{r_3}}
{\binom{x^{p,q}_{k}+k}{p-r_1}\binom{x^{a,b}_{p}+p}{a}\binom{x^{b,c}_{q}+q}{b-r_3}}
=
\frac{\lambda}{\binom{x^{a,c}_{k}+k}{a}}.
\]
Replacing $x_3$ by $x_q^{b,c}-r_3$ and using $\binom{s}{t}=\binom{s}{s-t}$ at appropriate places gives
\[
(-1)^{x^{a,b}_{p}+x^{b,c}_{q}}
\sum_{r_1,r_3}
(-1)^{r_3}
\frac{\binom{x^{p,q}_{k}-r_1}{x^{a,c}_{k}-r_3}
\binom{x^{p,q}_{k}}{r_1}\binom{x^{b,c}_{q}}{r_3}}
{\binom{x^{p,q}_{k}+k}{p-r_1}\binom{x^{b,c}_{q}+q}{c-r_3}}
=
\lambda\frac{\binom{x^{a,b}_{p}+p}{a}}{\binom{x^{a,c}_{k}+k}{a}}.
\]
We next concentrate on the inner part of the above double sum.
\begin{lemma}
 \[
\sum_{r_1}
\frac{\binom{x^{p,q}_{k}-r_1}{x^{a,c}_{k}-r_3}\binom{x^{p,q}_{k}}{r_1}}
{\binom{x^{p,q}_{k}+k}{p-r_1}}
=
\frac{x^{p,q}_{k}+k+1}{x^{p,q}_{k}+k+1-p}\frac{\binom{x^{p,q}_{k}}{x^{a,c}_{k}-r_3}}{\binom{x^{a,c}_{k}+k+1-r_3}{x^{p,q}_{k}+k+1-p}}.
\]
\end{lemma}
\begin{proof}
Note that if $x$, $y$, $z$ are non-negative integers and $y\ge z$ then
\begin{equation}
\label{wall}
\sum_{r=0}^z\binom{x+r}{r}\binom{y-r}{z-r}=\binom{x+y+1}{z}.
\end{equation}
This is easily seen by induction on $y$ by repeatedly using $\binom{s}{t}=\binom{s-1}{t}+\binom{s-1}{t-1}$.
Using (\ref{wall}) in the equivalent form
\[
\sum_{r=0}^z\frac{(x+r)!\,(y-r)!}{r!\,(z-r)!}=\frac{x!\,(y-z)!\,(x+y+1)!}{z!\,(x+y-z+1)!}
\]
with $x=x_k^{p,q}+k-p$, $y=p$, $z=x^{p,q}_k-x^{a,c}_k+r_3$ yields
\begin{align*}
\sum_{r_1}
\frac{\binom{x^{p,q}_{k}-r_1}{x^{a,c}_{k}-r_3}\binom{x^{p,q}_{k}}{r_1}}
{\binom{x^{p,q}_{k}+k}{p-r_1}}
&=
\sum_{r_1}
\frac{(x^{p,q}_{k})!\,(x^{p,q}_{k}+k-p+r_1)!\,(p-r_1)!}
{(x^{a,c}_{k}-r_3)!\,(x^{p,q}_{k}-r_1-x^{a,c}_{k}+r_3)!\,r_1!\,(x^{p,q}_{k}+k)!} \\
&=
\frac{(x^{p,q}_{k})!}{(x^{a,c}_{k}-r_3)!\,(x^{p,q}_{k}+k)!}
\sum_{r_1}
\frac{(x^{p,q}_{k}+k-p+r_1)!\,(p-r_1)!}
{r_1!\,(x^{p,q}_{k}-x^{a,c}_{k}+r_3-r_1)!} \\
&=
\frac{(x^{p,q}_{k})!}{(x^{a,c}_{k}-r_3)!\,(x^{p,q}_{k}+k)!}\\
&\qquad\times
\frac{(x^{p,q}_{k}+k-p)!\,(p-x^{p,q}_{k}+x^{a,c}_{k}-r_3)!\,(x^{p,q}_{k}+k+1)!}
{(x^{p,q}_{k}-x^{a,c}_{k}+r_3)!\,(k+x^{a,c}_{k}-r_3+1)!}  \\
&=
\frac{x^{p,q}_{k}+k+1}{x^{p,q}_{k}+k+1-p}
\frac{\binom{x^{p,q}_{k}}{x^{a,c}_{k}-r_3}}{\binom{x^{a,c}_{k}+k+1-r_3}{x^{p,q}_{k}+k+1-p}}
\end{align*}
as we wanted to prove.
\end{proof}
From the above lemma we obtain
\[
(-1)^{x^{a,b}_{p}+x^{b,c}_{q}}\frac{x^{p,q}_{k}+k+1}{x^{p,q}_{k}+k+1-p}
\sum_{r_3}
(-1)^{r_3}
\frac{\binom{x^{p,q}_{k}}{x^{a,c}_{k}-r_3}\binom{x^{b,c}_{q}}{r_3}}{\binom{x^{a,c}_{k}+k+1-r_3}{x^{p,q}_{k}+k+1-p}\binom{x^{b,c}_{q}+q}{c-r_3}}
=
\lambda\frac{\binom{x^{a,b}_{p}+p}{a}}{\binom{x^{a,c}_{k}+k}{a}}.
\]

We refer the reader to \S\ref{Subsec.Clebsch-Gordan} for the definition and basic properties of the $6j$-symbol
as well as the meaning of $\Delta$ and $R$ as used below.

\begin{theorem}\label{Thm.Scalar for phi}
\[
\lambda=
C\,
\left\{\begin{matrix}
\frac{q}{2} \; \frac{k}{2} \; \frac{p}{2} \\[2mm]
\frac{a}{2} \; \frac{b}{2} \; \frac{c}{2}
\end{matrix}
\right\},
\]
where
\[
 C=
\tfrac{(-1)^{x_k^{a,c}+b+k}(p+q+k+2)(a+b+p+2)(b+c+q+2)
\Delta(\frac{a}{2},\frac{b}{2},\frac{p}{2})\,\Delta(\frac{p}{2},\frac{q}{2},\frac{k}{2})\,\Delta(\frac{b}{2},\frac{c}{2},\frac{q}{2})}
{4(a+c+k+2)\,\Delta(\frac{a}{2},\frac{c}{2},\frac{k}{2})}.
\]
\end{theorem}
\begin{proof}
From the identity above the theorem we have that
\[
\lambda=
(-1)^{x^{a,b}_{p}+x^{b,c}_{q}}\tfrac{x^{p,q}_{k}+k+1}{x^{p,q}_{k}+k+1-p}
\tfrac{(x^{a,c}_{k}+k)!\,(x^{a,b}_{p}+p-a)!}{(x^{a,c}_{k}+k-a)!\,(x^{a,b}_{p}+p)!}
\sum_{r_3}
(-1)^{r_3}
\frac{\binom{x^{p,q}_{k}}{x^{a,c}_{k}-r_3}\binom{x^{b,c}_{q}}{r_3}}{\binom{x^{a,c}_{k}+k+1-r_3}{x^{p,q}_{k}+k+1-p}\binom{x^{b,c}_{q}+q}{c-r_3}}.
\]
If we replace $r_3$ by $x^{a,c}_{k}-t$, the above sum is
\begin{align*}
\sum_{r_3}
(-1)^{r_3}
&
\frac{\binom{x^{p,q}_{k}}{x^{a,c}_{k}-r_3}\binom{x^{b,c}_{q}}{r_3}}{\binom{x^{a,c}_{k}+k+1-r_3}{x^{p,q}_{k}+k+1-p}\binom{x^{b,c}_{q}+q}{c-r_3}}
=
\tfrac{x^{p,q}_{k}!\,x^{b,c}_{q}!\,(x^{p,q}_{k}+k+1-p)!}{(x^{b,c}_{q}+q)!} \\
&\times
\sum_{r_3}
\tfrac{(-1)^{r_3}(x^{a,c}_{k}-x^{p,q}_{k}+p-r_3)!\,(c-r_3)!\,(x^{b,c}_{q}+q-c+r_3)!\,}
{(x^{a,c}_{k}-r_3)!\,(x^{p,q}_{k}-x^{a,c}_{k}+r_3)!\;r_3!\,(x^{b,c}_{q}-r_3)!\,(x^{a,c}_{k}+k+1-r_3)!}  \\
&=
(-1)^{x^{a,c}_{k}}\;\tfrac{x^{p,q}_{k}!\,x^{b,c}_{q}!\,(x^{p,q}_{k}+k+1-p)!}{(x^{b,c}_{q}+q)!} \\
&\times
\sum_{t}
\tfrac{(-1)^{t}(-x^{p,q}_{k}+p+t)!\,(c-x^{a,c}_{k}+t)!\,(x^{b,c}_{q}+q-c+x^{a,c}_{k}-t)!\,}
{t!\,(x^{p,q}_{k}-t)!\,(x^{a,c}_{k}-t)!\,(x^{b,c}_{q}-x^{a,c}_{k}+t)!\,(k+1+t)!}.
\end{align*}
Define $j_1$, $j_2$, $j_3$, $j_4$, $j_5$, $j_6$  by
\begin{align*}
  a&=2j_1,
 &p&=2j_2,
 &b&=2j_3, \\
 q&=2j_4,
 &c&=2j_5,
 &k&=2j_6.
\end{align*}
Then above sum together with (\ref{eqn.def6j}) and $a+c\equiv k\mod 2$ yield
\begin{align*}
\sum_{r_3}
(-1)^{r_3}
&
\frac{\binom{x^{p,q}_{k}}{x^{a,c}_{k}-r_3}\binom{x^{b,c}_{q}}{r_3}}{\binom{x^{a,c}_{k}+k+1-r_3}{x^{p,q}_{k}+k+1-p}\binom{x^{b,c}_{q}+q}{c-r_3}} \\
&=
(-1)^{j_1+j_5-j_6}\;\tfrac{(j_2+j_4-j_6)!\,(j_3-j_4+j_5)!\,(-j_2+j_4+j_6+1)!}{(j_3+j_4+j_5)!} \\
&\qquad \times
\sum_{t}
\tfrac{(-1)^{t}(j_2-j_4+j_6+t)!\,(-j_1+j_5+j_6+t)!\,(j_1+j_3+j_4-j_6-t)!\,}
{t!\,(j_2+j_4-j_6-t)!\,(j_1+j_5-j_6-t)!\,(-j_1+j_3-j_4+j_6+t)!\,(2j_6+1+t)!} \\
&=
(-1)^{j_2+j_4+j_6}\;\tfrac{(j_2+j_4-j_6)!\,(j_3-j_4+j_5)!\,(-j_2+j_4+j_6+1)!}{(j_3+j_4+j_5)!}
\tfrac{R_{5,6}^{1}R_{2,6}^{4}}{R_{2,3}^{1}R_{3,5}^{4}}
\left\{\begin{matrix}
j_1 \; j_2 \; j_3 \\
j_4 \; j_5 \; j_6
\end{matrix}
\right\}.
\end{align*}
Therefore
\begin{align*}
\lambda=
(-1)^{j_1+2j_3+j_5+j_6}
&
\tfrac{j_2+j_4+j_6+1}{-j_2+j_4+j_6+1}
\tfrac{(j_1+j_5+j_6)!\, (-j_1+j_2+j_3)!}{(-j_1+j_5+j_6)!\, (j_1+j_2+j_3)!}  \\
&\times
\tfrac{(j_2+j_4-j_6)!\,(j_3-j_4+j_5)!\,(-j_2+j_4+j_6+1)!}{(j_3+j_4+j_5)!}
\tfrac{R_{5,6}^{1}R_{2,6}^{4}}{R_{2,3}^{1}R_{3,5}^{4}}
\left\{\begin{matrix}
j_1 \; j_2 \; j_3 \\
j_4 \; j_5 \; j_6
\end{matrix}
\right\}
                                                          \\
=(-1)^{j_1+2j_3+j_5+j_6}
&
\tfrac{(j_2+j_4+j_6+1)(j_1+j_2+j_3+1)(j_3+j_4+j_5+1)}{(j_1+j_5+j_6+1)}
\tfrac{\Delta_{123}\Delta_{246}\Delta_{345}}{\,\Delta_{156}}
\left\{\begin{matrix}
j_1 \; j_2 \; j_3 \\
j_4 \; j_5 \; j_6
\end{matrix}
\right\}.
\end{align*}
The theorem now follows from the symmetry
$
\left\{\begin{matrix}
j_1 \; j_2 \; j_3 \\
j_4 \; j_5 \; j_6
\end{matrix}
\right\}
=
\left\{\begin{matrix}
j_4 \; j_6 \; j_2 \\
j_1 \; j_3 \; j_5
\end{matrix}
\right\} $.
\end{proof}

\section{Appendix. The Clebsch-Gordan coefficients and the $6j$-symbol}
\label{Subsec.Clebsch-Gordan}
In this appendix we recall the basic facts about the $6j$-symbol the we needed in this paper.
We will mainly follow \cite{VMK}.

Let $2j_1$, $2j_2$ and $2j_3$ be three non-negative integers and define (see \cite[\S8.2, eq.(1)]{VMK})
\[
 \Delta(j_1,j_2,j_3)=\sqrt{\frac{(j_1+j_2-j_3)!(j_1-j_2+j_3)!(-j_1+j_2+j_3)!}{(j_1+j_2+j_3+1)!}}
 \]
if $(2j_1, 2j_2, 2j_3)$  satisfies the triangle condition
(see Definition \ref{def.triangle});
otherwise set $\Delta(j_1,j_2,j_3)=0$.

If $2m_1$, $2m_2$ and $2m_3$ are three integers such that $|m_i|\le j_i$ for $i=1,2,3$,
we recall that the
corresponding \emph{Clebsch-Gordan coefficient} is zero, if $m_1+m_2\ne m_3$,
and otherwise is (see \cite[\S8.2, eq.(3)]{VMK})
 \begin{align*}
 C^{j_3,m_3}_{j_1,m_1;\,j_2,m_2}&=\Delta(j_1,j_2,j_3)\sqrt{(2j_3+1) } \\[2mm]
 &
 \times \sqrt{(j_1+m_1)!(j_1-m_1)!(j_2+m_2)!(j_2-m_2)!(j_3+m_3)!(j_3-m_3)! } \\[2mm]
&
 \times
 \sum_r\tfrac{(-1)^r}{r!(j_1+j_2-j_3-r)!(j_1-m_1-r)!(j_2+m_2-r)!(j_3-j_2+m_1+r)!(j_3-j_1-m_2+r)!},
 \end{align*}
where the sum runs over all $r$ such that all the numbers under the factorial symbol are non-negative.

Let $a=2j_1$, $b=2j_2$ and $k=2j$.
If we define (see \cite[\S3.1.1]{VMK})
\[
\mathcal{M}_{j,m}=\sqrt{\frac{(j+m)!}{(j-m)!}}\;F^{j-m}e_k
\]
then (see \cite[\S8.2, eq.(10)]{VMK})
{\small
\begin{equation*}
\left\{
\{\mathcal{M}_{j_1}\otimes\mathcal{M}_{j_2}\}_{j,m}:=
\sum_{m_1+m_2=m} C^{j,m}_{j_1,m_1;\;j_2,m_2}\;
\mathcal{M}_{j_1,m_1}\!\otimes\! \mathcal{M}_{j_2,m_2}:m=-j,-j+1,\dots,j
\right\}
\end{equation*}
}is a basis of the unique $\sl(2)$-submodule of $V(a)\otimes V(b)$ isomorphic to $V(k)$,
and in fact, the map
\begin{align*}
V(k) & \to V(a)\otimes V(b) \\
\mathcal{M}_{j,m} &\mapsto\{\mathcal{M}_{j_1}\otimes\mathcal{M}_{j_2}\}_{j,m}
\end{align*}
is exactly
$\frac{1}{\Delta(j_1,j_2,j)}\sqrt{\frac{2j+1}{j_1+j_2+j+1}}\;\iota_k^{a,b}$.

\medskip

The (classical) \emph{Racah-Wigner $6j$-symbol} is a real number
$\left\{
\begin{matrix}
j_1 \;  j_2 \;  j_3 \\
j_4 \;  j_5 \;  j_6
\end{matrix}
\right\}$
associated
to six
non-negative
half-integer numbers $j_1$, $j_2$, $j_3$, $j_4$, $j_5$ and $j_6$.
The $6j$-symbol plays a central role in angular momentum theory since they describe
the recoupling of three angular momenta. Some classical references to them,
other than \cite{VMK}, are \cite{CFS}, \cite{Ed}, \cite{RBMW}, etc.
Let us recall its definition in terms of
the representation theory of $\sl(2)$.

If one of following four triples
\[
 (2j_1, 2j_2, 2j_3),\; (2j_1, 2j_5, 2j_6),\; (2j_4, 2j_2, 2j_6),\; (2j_4, 2j_5, 2j_3)
 \]
does not satisfy the triangle condition then
$\left\{
\begin{matrix}
j_1 \;  j_2 \;  j_3 \\
j_4 \;  j_5 \;  j_6
\end{matrix}
\right\}$ is zero by definition.
If all the above four triples do satisfy the triangle condition,
which may be depicted by the following tetrahedron,
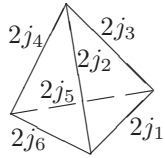
\setlength{\unitlength}{10pt}
\begin{center}
\begin{picture}(10,7)(0,-2)
 \put(0,0){\line(1,2){2.1}}
 \put(0,0){\line(2,-1){3}}
\put(3,-1.5){\line(1,1){2.4}}
\put(3.05,-1.5){\line(-1,6){.95}}
 \put(2.1,4.2){\line(1,-1){3.3}}
 \multiput(0,0)(1.45,.24){4}{\line(6,1){1}}
 \put(3.4,3.0){$2j_3$}%
 \put(-0.1,2.7){$2j_4$}
 \put(2.5,1.8){$2j_2$}%
 \put(4.5,-0.8){$2j_1$}
 \put(0.0,-1.2){$2j_6$}%
 \put(1.1,0.6){$2j_5$}%
\end{picture}
\end{center}
let
\begin{align}\label{Rel.abc...j}
  a&=2j_1,
 &b&=2j_2,
 &p&=2j_3, \\\notag
 c&=2j_4,
 &k&=2j_5,
 &q&=2j_6.
\end{align}
The $6j$-symbols
$\left\{
\begin{matrix}
j_1 \;  j_2 \;  j_3 \\
j_4 \;  j_5 \;  j_6
\end{matrix}
\right\}$
are the coefficients needed to express the $\sl(2)$-module homomorphism
\[
\big(\iota^{a,b}_{p}\otimes 1 \big)\circ\iota^{p,c}_{k}:V(k)\to V(p)\otimes V(c)\to V(a)\otimes V(b)\otimes V(c)
\]
as a linear combination of the $\sl(2)$-module homomorphisms
\[
\big(1 \otimes \iota^{b,c}_{q}\big)\circ\iota^{a,q}_{k}:V(k)\to V(a)\otimes V(q)\to V(a)\otimes V(b)\otimes V(c)
\]
as $q$ varies. More precisely, the following sets
\[
 \big\{\big(\iota^{a,b}_{p}\otimes 1 \big)\circ\iota^{p,c}_{k}:p\in\Z\text{ and }
\big(\iota^{a,b}_{p}\otimes 1 \big)\circ\iota^{p,c}_{k}\ne0 \big\}
\]
\[
 \big\{\big(1 \otimes \iota^{b,c}_{q}\big)\circ\iota^{a,q}_{k}:q\in\Z\text{ and }
\big(1 \otimes \iota^{b,c}_{q}\big)\circ\iota^{a,q}_{k}\ne0 \big\}
\]
are two different bases of $\text{Hom}_{\sl(2)}\big(V(k),V(a)\otimes V(b)\otimes V(c)\big)$
and the $6j$-symbol describe the transition matrix between these two bases, that is
$\left\{
\begin{matrix}
j_1 \; j_2 \; j_3 \\
j_4 \; j_5 \; j_6
\end{matrix}
\right\}$
is implicitly defined by the following identity
(see \eqref{Rel.abc...j})
\begin{align*}
 &\tfrac{(-1)^{j_1-j_2-j_4+j_5}}
{(j_1+j_2+j_3+1)(j_3+j_4+j_5+1)\Delta(j_1,j_2,j_3)\Delta(j_3,j_4,j_5)}
 \big(\iota^{a,b}_{p}\otimes 1 \big)\circ\iota^{p,c}_{k} \\
&=
\sum_{q}
\tfrac{(-1)^{q}(q+1)}
{(j_2+j_4+j_6+1)(j_1+j_6+j_5+1)\Delta(j_2,j_4,j_6)\Delta(j_1,j_6,j_5)}
\left\{
\begin{matrix}
j_1 \; j_2 \; j_3 \\
j_4 \; j_5 \; j_6
\end{matrix}
\right\}
\big(1 \otimes \iota^{b,c}_{q}\big)\circ\iota^{a,q}_{k}.
\end{align*}
which is equivalent to say that
\begin{align*}
 \frac{(-1)^{j_1-j_2-j_4+j_5}}{\sqrt{2j_3+1}}
 &
 C^{j_3,m_3}_{j_1,m_1;\;j_2,m_2} C^{j_5,m_5}_{j_3,m_3;\;j_4,m_4} \\
&=
\sum_{j_6}(-1)^{2j_6}\sqrt{2j_6+1}
\left\{\begin{matrix}
j_1 \; j_2 \; j_3 \\
j_4 \; j_5 \; j_6
\end{matrix}
\right\}
 C^{j_6,m_6}_{j_2,m_2;\;j_4,m_4}
 C^{j_5,m_5}_{j_1,m_1;\;j_6,m_6}
\end{align*}
for all $m_i$, $i=1,\dots,6$, such that  $|m_i|\le j_i$ and
$m_1+m_2=m_3$, $m_3+m_4=m_5$, $m_2+m_4=m_6$ and $m_1+m_6=m_5$.
This identity is derived from \cite[\S8.7.5, eq.(36)]{VMK} and the
symmetry properties of the Clebsch-Gordan coefficients \cite[\S8.4.2, eq.(5)]{VMK}.

If we set
$
\Delta_{x,y,z}=\Delta(j_x,j_y,j_z),
$
then the $6j$-symbol can be explicitly expressed as
(see \cite[\S9.2.1, eq.(1)]{VMK})
\begin{align*}
\left\{\begin{matrix}
j_1 \; j_2 \; j_3 \\
j_4 \; j_5 \; j_6
\end{matrix}
\right\}
&=\Delta_{1,2,3}\Delta_{3,4,5}\Delta_{2,4,6}\Delta_{1,5,6}\\
&\times
\sum_{t}
\frac{(-1)^t(t+1)!}{(t-\alpha_0)!\,(t-\alpha_1)!\,(t-\alpha_2)!\,(t-\alpha_3)!\,(\beta_1-t)!\,(\beta_2-t)!\,(\beta_3-t)!}
\end{align*}
where $t$ runs from $\max\{\alpha_0,\alpha_1,\alpha_2,\alpha_3\}$ to $\min\{\beta_1,\beta_2,\beta_3\}$ and
\begin{align*}
\alpha_0&=j_1+j_2+j_3, &       &                 \\
\alpha_1&=j_1+j_5+j_6, &\beta_1&=j_2+j_3+j_5+j_6, \\
\alpha_2&=j_4+j_2+j_6, &\beta_2&=j_1+j_3+j_4+j_6, \\
\alpha_3&=j_4+j_5+j_3, &\beta_3&=j_1+j_2+j_4+j_5.
\end{align*}
Also, if
\[
R_{x,y}^{z}=\sqrt{\tfrac{(j_x+j_y-j_z)!}{(j_x-j_y+j_z)!\,(-j_x+j_y+j_z)!\,(j_x+j_y+j_z+1)!}}.
\]
then (see \cite[\S9.2.1, eq.(5)]{VMK})
\begin{align}\label{eqn.def6j}
  \left\{\begin{matrix}
j_1 \; j_2 \; j_3 \\
j_4 \; j_5 \; j_6
\end{matrix}
\right\}
&=(-1)^{j_1+j_2+j_4+j_5}
\frac{R_{2,3}^{1}R_{3,5}^{4}}{R_{5,6}^{1}R_{2,6}^{4}} \\ \notag
\times\sum_{t}&
\frac{(-1)^t\,(-j_1+j_5+j_6+t)!\,(j_2-j_4+j_6+t)!\,(j_1+j_3+j_4-j_6-t)!}
{t!\,(j_1\!+\!j_5\!-\!j_6\!-\!t)!\,(j_2\!+\!j_4\!-\!j_6\!-\!t)!\,(-j_1\!+\!j_3\!-\!j_4\!+\!j_6\!+\!t)!\,(2j_6\!+\!1\!+\!t)!}.
\end{align}
We need the following three properties of the $6j$-symbol (see \cite[\S9.4.2]{VMK}):
\begin{enumerate}[(i)]
\item \label{item.6j1}The $6j$-symbol is invariant under the permutation of any two columns.
\item \label{item.6j2}The $6j$-symbol is invariant if upper and lower arguments are interchanged in any two columns.
\item \label{item.6j3}If all the triples
\[
 (2j_1, 2j_2, 2j_3),\; (2j_1, 2j_5, 2j_6),\; (2j_4, 2j_2, 2j_6),\; (2j_4, 2j_5, 2j_3)
 \]
satisfy the triangle condition, but one of them is a degenerate triangle, then
$  \left\{\begin{matrix}
j_1 \; j_2 \; j_3 \\
j_4 \; j_5 \; j_6
\end{matrix}
\right\}\ne0$.
Indeed, if one of the above triples corresponds to a degenerate triangle, then
\eqref{item.6j1} and \eqref{item.6j2} imply that we may assume $j_6=j_1+j_5$.
Now it follows  from \eqref{eqn.def6j} that
\begin{align*}
  \left\{\begin{matrix}
j_1 \; j_2 \; j_3 \\
j_4 \; j_5 \; j_6
\end{matrix}
\right\}
&=(-1)^{j_1+j_2+j_4+j_5}
\frac{R_{2,3}^{1}R_{3,5}^{4}}{R_{5,6}^{1}R_{2,6}^{4}} \\
&
\frac{(-j_1+j_5+j_6)!\,(j_2-j_4+j_6)!\,(j_1+j_3+j_4-j_6)!}
{(j_2\!+\!j_4\!-\!j_6\!)!\,(-j_1\!+\!j_3\!-\!j_4\!+\!j_6\!)!\,(2j_6\!+\!1\!)!}\ne0.
 \end{align*}
\end{enumerate}
The following lemma shows that, under certain additional conditions,
other $6j$-symbols are non-zero.

\begin{lemma}\label{lemma.non-zero}
Let   $j_1$, $j_2$, $j_3$, $j_4$, $j_5$ and $j_6$ be
non-negative
half-integer such that $j_6=j_1+j_5$, $j_2=j_3$ and all the triples
\[
 (2h, 2j_2, 2j_3),\; (2h, 2j_5, 2j_6),\; (2j_4, 2j_2, 2j_6),\; (2j_4, 2j_5, 2j_3)
 \]
satisfy the triangle condition for  $h=j_1$ and $h=j_1+1$.
If
$
  \left\{\begin{matrix}
j_1\!+\!1 \; j_2 \; j_3 \\
\;\;j_4 \;\; \; j_5 \; j_6
\end{matrix}
\right\}=0
$
then
$
  \left\{\begin{matrix}
j_1\!+\!2 \; j_2 \; j_3 \\
\;\;j_4 \;\; \; j_5 \; j_6
\end{matrix}
\right\}\ne0
$
and
$
  \left\{\begin{matrix}
j_1\!+\!3 \; j_2 \; j_3 \\
\;\;j_4 \;\; \; j_5 \; j_6
\end{matrix}
\right\}\ne0
$.
\end{lemma}
\begin{proof}
The Biedenharn-Elliott identity yields, in particular, the following
three-term recurrence relation (see \cite[pag. 1963]{SG})
\begin{align}\label{eq.BE}
i_1E(i_1+1)  \left\{\begin{matrix}
i_1\!\!+\!\!1 \; i_2 \; i_3 \\
\;\; i_4 \;\; i_5 \; i_6
\end{matrix}
\right\}
+F(i_1)  \left\{\begin{matrix}
i_1 \; i_2 \; i_3 \\
i_4 \; i_5 \; i_6
\end{matrix}
\right\}
+(i_1+1)E(i_1)  \left\{\begin{matrix}
i_1\!\!-\!\!1 \; i_2 \; i_3 \\
\;\; i_4\;\; i_5 \; i_6
\end{matrix}
\right\}=0
\end{align}
where
\begin{align*}
F(i_1)=
(2i_1 + 1)\big(&
 i_1(i_1+1)(-i_1(i_1+1) + i_2(i_2+1) + i_3(i_3+1)) \\
 + &i_5(i_5+1)( i_1(i_1+1) + i_2(i_2+1) - i_3(i_3+1)) \\
 + &i_6(i_6+1)( i_1(i_1+1) - i_2(i_2+1) + i_3(i_3+1)) \\
 - &2i_1(i_1+1)i_4(i_4+1)
\big)
\end{align*}
and
\begin{align*}
E(i_1)\!=\!
 \sqrt{\big(i_1^2 - (i_2-i_3)^2\big)\big((i_2+i_3+1)^2 - i_1^2\big)
 \big(i_1^2 - (i_5-i_6)^2\big)\big((i_5+i_6+1)^2 - i_1^2\big)}.
\end{align*}

If we fix $(i_2,i_3,i_4,i_5,i_6)=(j_2,j_3,j_4,j_5,j_6)$ we obtain
\begin{align*}
E(i_1)&=
 \sqrt{i_1^2 \big((2j_2+1)^2 - i_1^2\big)\big(i_1^2 - j_1^2\big)\big((j_1+2j_5+1)^2 - i_1^2\big)} \\
 F(i_1)&=
-(2i_1 + 1)i_1(i_1+1) \\
&\hspace{1cm}\times(
 i_1(i_1+1) - 2j_2(j_2+1) - j_5(j_5+1) - j_6(j_6+1) + 2j_4(j_4+1)
),
\end{align*}
and we point out that the triangle conditions satisfied by
$(2(j_1+1), 2j_5, 2j_6)$ and $(2(j_1+1), 2j_2, 2j_3)$
imply that
\[
E(j_1+1)\ne0.
\]
We also claim that  $F(j_1+2)\ne0$, and this will be proved later by
considering separately the cases $j_1=0$ and $j_1>0$.

Assume that
$
 \left\{\begin{matrix}
j_1\!+\!1 \; j_2 \; j_3 \\
\;\;j_4 \;\; \; j_5 \; j_6
\end{matrix}
\right\}= 0
$.
Since $j_6=j_5+j_1$ it follows that
$\left\{\begin{matrix}
j_1 \; j_2 \; j_3 \\
j_4 \; j_5 \; j_6
\end{matrix}
\right\}\ne0$
(see \eqref{item.6j3} above)
and
$\left\{\begin{matrix}
j_1\!-\!1 \; j_2 \; j_3 \\
\;\;j_4 \;\; \; j_5 \; j_6
\end{matrix}
\right\}= 0$ (the triple $(2(j_1-1),2j_5,2j_6)$ does not satisfy the triangle condition).
Since $E(j_1+1)\ne0$,
it follows, from
the recurrence relation \eqref{eq.BE}
applied to $i_1=j_1+1$, that
$  \left\{\begin{matrix}
j_1\!+\!2 \; j_2 \; j_3 \\
\;\;j_4 \;\; \; j_5 \; j_6
\end{matrix}
\right\}\ne 0$.

Now, accepting that $F(j_1+2)\ne0$,
the recurrence relation \eqref{eq.BE}
applied to $i_1=j_1+2$ implies that
$\left\{\begin{matrix}
j_1\!+\!3 \; j_2 \; j_3 \\
\;\;j_4 \;\; \; j_5 \; j_6
\end{matrix}
\right\}\ne0$ as we wanted to prove.

It remains to be proved that $F(j_1+2)\ne0$.
From the recurrence relation \eqref{eq.BE}
applied to $i_1=j_1$ we obtain that $F(j_1)=0$
and, if $j_1>0$, it follows that
\begin{equation*}\label{eq.Cond1}
  j_1(j_1+1) - 2j_2(j_2+1) - j_5(j_5+1) - j_6(j_6+1) + 2j_4(j_4+1)=0,
\end{equation*}
which implies that $F(j_1+2)\ne0$.
If $j_1=0$ then
\begin{align*}
0&=\left\{\begin{matrix}
j_1\!+\!1 \; j_2 \; j_3 \\
\;\;j_4 \;\; \; j_5 \; j_6
\end{matrix}
\right\} \\
&=\left\{\begin{matrix}
1 \; j_2 \; j_2 \\
j_4 \; j_6 \; j_6
\end{matrix}
\right\} \\
&=(-1)^{1+j_2+j_4+j_6}
\frac{R_{2,3}^{1}R_{3,5}^{4}}{R_{5,6}^{1}R_{2,6}^{4}} \\
& \times\sum_{t=0}
\frac{(-1)^t\,(-1+2j_6+t)!\,(j_2-j_4+j_6+t)!\,(1+j_2+j_4-j_6-t)!}
{t!\,(1-t)!\,(j_2+j_4-j_6-t)!\,(-1+j_2-j_4+j_6+t)!\,(2j_6+1+t)!}.
\end{align*}
and hence
\begin{align*}
& \frac{(-1+2j_6)!\,(j_2-j_4+j_6)!\,(1+j_2+j_4-j_6)!}
{(j_2+j_4-j_6)!\,(-1+j_2-j_4+j_6)!\,(2j_6+1)!} \\
&\hspace{2cm}-
 \frac{(2j_6)!\,(j_2-j_4+j_6+1)!\,(j_2+j_4-j_6)!}
{(j_2+j_4-j_6-1)!\,(j_2-j_4+j_6)!\,(2j_6+2)!}=0,
\end{align*}
or
\begin{align*}
 {(j_2-j_4+j_6)(1+j_2+j_4-j_6)(j_6+1)} -
 {(j_2-j_4+j_6+1)(j_2+j_4-j_6)j_6}=0,
\end{align*}
which implies $j_2(j_2+1)+j_6(j_6+1) - j_4(j_4+1)=0$
and therefore
\begin{align*}
 F(2)= -30\big(6 - 2j_2(j_2+1) - j_5(j_5+1) - j_6(j_6+1) + 2j_4(j_4+1)\big) = -180.
\end{align*}
\end{proof}


\end{document}